\documentclass[smallextended,referee,envcountsect,]{svjour3}
\smartqed
\usepackage{graphicx}
\usepackage{hyperref}
\hypersetup{
    colorlinks=true,
    linkcolor=blue,
    filecolor=magenta,      
    urlcolor=cyan,
    pdftitle={XXXXXXX},
    pdfpagemode=FullScreen,
    }
\journalname{Noname}
\usepackage{mathtools}
\usepackage[normalem]{ulem} % to use \sout

\usepackage{amsmath}
\usepackage{amssymb}
\usepackage{mathrsfs}
\usepackage{empheq}
\usepackage{hhline}
\usepackage{tikz}
\usepackage{tikz-cd}
\usepackage{subfig}
\newcommand{\squareop}[1]{%
    \setlength{\fboxsep}{0pt}%
    \setlength{\unitlength}{.7em}%
    \mathrel{%
        \raisebox{-1pt}{\framebox(1,1){\(\scriptstyle #1\)}}%
    }%
}

\newcommand{\overbar}[1]{\mkern 1.5mu\overline{\mkern-1.5mu#1\mkern-1.5mu}\mkern 1.5mu}

\newcommand{\cut}[1]{{}}
\newcommand{\vb}{{\mathbf{b}}}

\newcommand{\vq}{{\mathbf{q}}}

\newcommand{\vs}{{\mathbf{s}}}

\newcommand{\vv}{{\mathbf{v}}}

\newcommand{\vx}{{\mathbf{x}}}
\newcommand{\vy}{{\mathbf{y}}}
\newcommand{\vz}{{\mathbf{z}}}

\newcommand{\vA}{{\mathbf{A}}}
\newcommand{\vB}{{\mathbf{B}}}

\newcommand{\vI}{{\mathbf{I}}}

\newcommand{\vK}{{\mathbf{K}}}

\newcommand{\vM}{{\mathbf{M}}}
\newcommand{\vN}{{\mathbf{N}}}

\newcommand{\RR}{\mathbb{R}}

\newcommand{\vzero}{\mathbf{0}}

\newcommand{\st}{{\text{s.t.}}} % subject to
\newcommand{\prox}{\mathbf{prox}}

\mathchardef\mhyphen="2D

\makeatletter
\let\@@span\span
\def\sp@n{\@@span\omit\advance\@multicnt\m@ne}
\makeatother

\DeclareMathOperator*{\Min}{minimize}

\DeclarePairedDelimiter{\dotp}{\langle}{\rangle}

\newcommand{\bc}{\begin{center}}
\newcommand{\ec}{\end{center}}

\newcommand{\bdm}{\begin{displaymath}}
\newcommand{\edm}{\end{displaymath}}

\newcommand{\beq}{\begin{equation}}
\newcommand{\eeq}{\end{equation}}

\newcommand{\bfl}{\begin{flushleft}}
\newcommand{\efl}{\end{flushleft}}

\newcommand{\bt}{\begin{tabbing}}
\newcommand{\et}{\end{tabbing}}

\newcommand{\beqn}{\begin{align}}
\newcommand{\eeqn}{\end{align}}

\newcommand{\beqs}{\begin{align*}} % no equation numbers
\newcommand{\eeqs}{\end{align*}}  % no equation numbers

\begin{document}

\title{On the improved conditions for some primal-dual algorithms}

%\subtitle{Using  the  LaTex Template}

\author{Yao Li \and Ming Yan}

\institute{ Yao Li \at Michigan State University\\
East Lansing, MI 48824, USA\\
liyao6@msu.edu
\and
Ming Yan \at Michigan State University\\
East Lansing, MI 48824, USA\\
myan@msu.edu
%           \and
%              Virtual Author,  Corresponding author  \at
%              University of Dreamland \\
%              Dreamland\\
%              author2@example.com
}

\date{Received: date / Accepted: date}
%The correct dates will be entered by the editor.

\maketitle

\begin{abstract}
The convex minimization of $f(\vx)+g(\vx)+h(\vA\vx)$ over $\RR^n$ with differentiable $f$ and linear operator $\vA: \RR^n\rightarrow \RR^m$, has been well-studied in the literature. By considering the primal-dual optimality of the problem, many algorithms are proposed from different perspectives such as monotone operator scheme and fixed point theory. In this paper, we start with a base algorithm to reveal the connection between several algorithms such as AFBA, PD3O and Chambolle-Pock. Then, we prove its convergence under a relaxed assumption associated with the linear operator and characterize the general constraint on primal and dual stepsizes. The result improves the upper bound of stepsizes of AFBA and indicates that Chambolle-Pock, as the special case of the base algorithm when $f=0$, can take the stepsize of the dual iteration up to $4/3$ of the previously proven one. 

\end{abstract}
%\subclass{49J53 \and  49K99 \and more}

%All acknowledgements should be placed in the back of the paper after Conclusions..

\section{Introduction}\label{sec:intro}
We consider the following minimization problem in the form of the sum of three functions
\begin{equation}\label{pb1}
    \Min_{
    \vx\in\RR^n}\ f(\vx)+g(\vx)+h(\vA\vx),
\end{equation}
where $f$ is a differentiable convex function with $L$-Lipschitz continuous gradient and $g$, $h$ are proper, closed and convex functions taking values in $(-\infty, \infty].$ The third function $h$ is composited with a linear operator $\vA\in \RR^{n\times m}$, whose largest singular value is $\sigma$, i.e., $\sigma=\sqrt{\|\vA\vA^\top\|}$. 

We use the asterisk superscript to denote the Legendre-Fenchel conjugate function, e.g.,  $h^*:\RR^m\rightarrow\RR$ is defined as
\begin{equation*}
    h^*(\vs)=\sup_{\vx\in\RR^m}\ \dotp{\vs,\vx}-h(\vx).
\end{equation*}
Then, the saddle-point form of problem~\eqref{pb1} is
\begin{equation}\label{pb2}
    \min_{\vx\in \RR^n}\max_{\vs\in\RR^m}\ f(\vx)+g(\vx)+\dotp{\vA\vx,\vs}-h^*(\vs)
\end{equation}
where $\vx$ and $\vs$ are primal and dual variables, respectively. With the existence of a solution pair $(\vx^*,\vs^*)$, the first-order optimal condition is characterized as
\begin{equation*}
    \left\{\begin{aligned}
    & 0\in\nabla f(\vx^*)+\partial g(\vx^*)+\vA^\top\vs^*,\\
    & 0\in\partial h^*(\vs^*)-\vA\vx^*,
    \end{aligned}\right.
\end{equation*}
where $\partial g(\vx)=\{\vv\in\RR^n\ |\ g(\vy)-g(\vx)\geq\dotp{\vv,\vy-\vx}, \forall \vy \in\RR^n\}$ is the subdifferential of $g$ at $\vx$. The duality theorem~\cite[Theorem 15]{rockafellar1974conjugate} asserts that $\vx^*$ is the solution of the problem~\eqref{pb1} and $\vs^*$ is the solution of the corresponding dual problem 
\begin{equation}\label{pb2: dual}
    \Min_{
    \vs\in\RR^m}\ (f^*\squareop{} g^*)(-\vA^\top\vs)+h^*(\vs),
\end{equation}
where $f^*\squareop{} g^*$ is the infimal convolution of $f^*$ and $g^*$ that is defined as $f^*\squareop{} g^*(\vx)=\inf_{\vy\in\RR^n} f^*(\vy)+g^*(\vx-\vy)$ with the property that $(f^*\squareop{}g^*)^*=f+g.$ Note that when $f=0$ (or $g=0$), the infimal convolution $f^*\squareop{}g^*$ boils down to $g^*$ (or $f^*$).

Existing primal-dual algorithms to solve the above saddle-point problem~\eqref{pb2} include Condat-Vu~\cite{condat2013primal,vu2013splitting},  Primal–Dual Fixed-Point algorithm(PDFP)~\cite{chen2016primal}, Asymmetric Forward–Backward-Adjoint splitting(AFBA)~\cite{latafat2017asymmetric} and Primal–Dual Three-Operator splitting(PD3O)~\cite{yan2018new}. When $f=0$, Condat-Vu and PD3O are reduced to Chambolle-Pock~\cite{chambolle2011first}. When $g=0,$ all algorithms except Condat-Vu are reduced to Proximal Alternating Predictor–Corrector(PAPC) or Primal–Dual Fixed-Point algorithm based on the Proximity Operator(PDFP$^2$O)~\cite{loris2011generalization,chen2013primal,drori2015simple}.
The conditions on their stepsizes to guarantee the convergence is associated with the singular value of the linear operator $\sigma$ and the Lipschitz constant $L$. With notations $\eta_{\text{p}}$ for the primal stepsize and $\eta_{\text{d}}$ for the dual stepsize, Table~\ref{table1} lists the conditions on stepsizes of the  aforementioned algorithms and their connections when either $f$ or $g$ vanishes.   
\begin{table}[!ht]
\centering
\begin{tabular}{lllllll}
\hline
                &  & $\eta_{\text{p}}, \eta_{\text{d}}$                                                      &  & $f=0$                  &  & $g=0$ \\ \hhline{=======}
Condat-Vu       &  & $\eta_{\text{p}}\eta_{\text{d}}\sigma^2+\eta_{\text{p}}L/2\leq 1$                                               &  & C-P(primal) &  &       \\
PDFP            &  & $\eta_{\text{p}}L/2<1, \eta_{\text{p}}\eta_{\text{d}}\sigma^2\leq 1$                                            &  &                        &  & PAPC  \\
AFBA            &  & $\eta_{\text{p}}\eta_{\text{d}}\sigma^2+\sqrt{\eta_{\text{p}}\eta_{\text{d}}}\sigma+\eta_{\text{p}}L\leq 2$ &  & C-P(dual)   &  & PAPC  \\
PD3O            &  & $\eta_{\text{p}}L/2<1, \eta_{\text{p}}\eta_{\text{d}}\sigma^2\leq 1$                                            &  & C-P(primal) &  & PAPC  \\ \hline
                &  &                                                                                         &  &                        &  &       \\ \hhline{=======}
Chambolle-Pock  &  & $\eta_{\text{p}}\eta_{\text{d}}\sigma^2\leq1$                                                    &  & \multicolumn{3}{c}{$f=0$}         \\
PAPC(PDFP$^2$O) &  & $\eta_{\text{p}}L/2<1$, $ \eta_{\text{p}}\eta_{\text{d}}\sigma^2\leq1$                                    &  & \multicolumn{3}{c}{$g=0$}         \\ \hline
\end{tabular}
\caption{The conditions on stepsizes for the primal-dual algorithms. C-P(primal) and C-P(dual) stand for Chambolle-Pock applied to the primal problem and the dual problem, respectively. In~\cite{li2021new}, PAPC can take a larger upper bound of $\eta_{\text{p}}\eta_{\text{d}}\sigma^2$ up to $4/3$.}
\label{table1}
\end{table}

The listed conditions in Table~\ref{table1} are all sufficient to guarantee the convergence of the associated algorithms, and some cannot be relaxed any further. For PAPC, the first constraint, $\eta_{\text{p}}L/2<1$, cannot be relaxed since it reduces to gradient descent method with stepsize $\eta_{\text{p}}$ when there is only one $f$, i.e., $h=0$. However, it turns out that the upper bound of $\eta_{\text{p}}\eta_{\text{d}}\sigma^2$ can be relaxed to $4/3$. In ~\cite{he2020optimal,li2021linear}, a special case of PAPC, linearized augumented Lagrangian method, and an application of PAPC on the decentralized optimization are shown to converge under the relaxed condition. The general convergence result of PAPC with relaxed condition is shown in~\cite{li2021new}. The $4/3$ bound is also shown to be tight in~\cite{he2020optimal,li2021new}. This result raises an open question, i.e., \textit{can the similar relaxation be applied to Chambolle-Pock?} The follow-up question is \textit{can this result be generalized to algorithms for the general problem~\eqref{pb1}?}. 

The first question is implicitly answered by some work. A generalized Chambolle-Pock is proposed in~\cite{he2021generalized} to achieve the relaxation and its equivalence to the canonical Chambolle-Pock is proved for the following primal-dual problem
\begin{equation*}
    \min_{\vx\in\RR^n}\max_{\vs\in\RR^m}\ g(\vx)+\dotp{\vA\vx,\vs}-\dotp{\vb,\vs}
\end{equation*}
which covers the general form~\eqref{pb2} with linear $h^*$. The necessity of the condition is also shown by a simple case.

In~\cite{he2020optimally}, a relaxed linearization parameter is considered in the linearized Alternating Direction Method of Multipliers(L-ADMM) so that a larger stepsize can be used in the linearized subproblem to converge faster. L-ADMM considers the following linearly constrained problem,
\begin{equation}\label{ladmm:primal}
    \begin{aligned}
    \Min_{\vx\in\RR^{n_1},\vy\in\RR^{n_2}}&\ \ \widetilde{f}(\vx)+\widetilde{g}(\vy)\\
    \st&\ \ \widetilde{\vA}\vx+\widetilde{\vB}\vy=\vb
    \end{aligned}
\end{equation}
where $\widetilde{f}:\RR^{n_1}\rightarrow\RR$, $  \widetilde{g}:\RR^{n_2}\rightarrow\RR$ are proper, closed and convex, $\widetilde{\vA}\in\RR^{m\times n_1}$,  $\widetilde{\vB}\in\RR^{m\times n_2}$ are two linear operators and $\vb\in\RR^m$.  When $\widetilde{f}=g^*$ and $\widetilde{g}=h^*$, the dual problem can be formulated via Lagrangian multiplier as
\begin{equation}\label{ladmm:dual}
    \Min_{\vs\in\RR^{m}}\ -\dotp{\vb,\vs}+g(\widetilde{\vA}^\top\vs)+h(\widetilde{\vB}^\top\vs).
\end{equation}

A special case of the above dual problem~\eqref{ladmm:dual} is $\vb=\vzero$, $\widetilde{\vA}=\vI$, and $\widetilde{\vB}^\top=\vA$. In this case, it reduces to the problem~\eqref{pb1} with $f=0.$ It can be easily shown that L-ADMM applied to this special case of the problem~\eqref{ladmm:primal} is equivalent to Chambolle-Pock applied to the corresponeding dual problem. Therefore, the result in~\cite{he2020optimally} suggests the improved condition on stepsizes of Chambolle-Pock, $\eta_{\text{p}}\eta_{\text{d}}\sigma^2<\frac{4}{3}.$ When $\vb\not=0$ in the special case, the result in~\cite{he2020optimally} also indicates that the relaxed conditon is extensible to the problem~\eqref{pb1} with a linear function $f$.

However, there is no result for general Lipschitz smooth $f$ in the literature. This work will address the second question for AFBA and PD3O by characterizing a more general upper bound for $\eta_{\text{p}}$ and $\eta_{\text{d}}$, and directly give an affirmative answer to the first question.

Throughout the rest of the paper, we assume that there exists a solution pair $(\vx^*,\vs^*)$ of the problem~\eqref{pb2}. % , i.e, $\vx^*$ is the solution of the problem~\eqref{pb1} and the strong duality holds. 
 We use $(\vI+\lambda\partial g)^{-1}(\vx)$ to represent the proximal operator of $g$ at $\vx$, which is defined as
$$\prox_{\lambda g}(\vx)\coloneqq\mathop{\arg\min}_{\vy\in \RR^n}\ g(\vy)+\frac{1}{2\lambda}\|\vy-\vx\|^2.$$
We use $\dotp{\cdot,\cdot}$ and $\|\cdot\|$ as the standard inner product and norm defined over $\RR^n$, and denote $\|\cdot\|_\vM=\sqrt{\dotp{\cdot,\vM(\cdot)}}$ as the (semi)norm associated with the given positive (semi)definite matrix $\vM\in \RR^{n\times n}$.

The rest of this paper is organized as follows. We first derive a base algorithm that is later shown to be an equivalent form of AFBA and illustrate the connection between AFBA, PD3O and Chambolle-Pock in Section~\ref{sec:algo}. We then prove the convergence of the base algorithm under the relaxed condition associated with the matrix $\vA$ in Section~\ref{sec:conv}. In the next section, we numerically show the performance of the algorithms under the proved weaker condition. 

\section{The Base Algorithm}\label{sec:algo}
For simplicity, we set $\eta_{\text{p}}=r$ and $\eta_{\text{d}}=\lambda/r$. That is, the product of the primal and dual stepsizes is $\lambda$. The optimality condition of problem~\eqref{pb2} can be reformulated as the following monotone inclusion problem 
\begin{equation}\label{pb3}
  \begin{bmatrix}
         0\\
         0
  \end{bmatrix}\in\begin{bmatrix}
        \nabla f(\vx) \\
        0
      \end{bmatrix}+
      %\begin{bmatrix}
                      %\partial g(\vx) \\
                     % 0
                    %\end{bmatrix}+
                    \begin{bmatrix}
                                    \partial g & \vA^\top \\
                                    -\vA & \partial h^*
                                  \end{bmatrix}\begin{bmatrix}
                                                 \vx \\
                                                 \vs
                                               \end{bmatrix}.
\end{equation}
This problem is also equivalent to  
\begin{equation*}
    \begin{bmatrix}
        \vI&\\
        &\frac{r^2}{\lambda}\vI-r^2\vA\vA^\top
     \end{bmatrix}\begin{bmatrix}
        \vx\\
        \vs
     \end{bmatrix}\in\begin{bmatrix}
      r \nabla f(\vx)\\
      0
     \end{bmatrix}+\begin{bmatrix}
                                    \vI+r\partial g & r\vA^\top \\
                                    -r\vA & \frac{r^2}{\lambda}\vI-r^2\vA\vA^\top+r\partial h^*
                                  \end{bmatrix}\begin{bmatrix}
                                                 \vx \\
                                                 \vs
                                               \end{bmatrix}.
\end{equation*}

Let $\zeta\in \vx-r\nabla f(\vx)-r\partial g(\vx)$, then the above form can be decomposed as
\begin{subequations}\label{pb3:sys}
    \begin{empheq}[left=\empheqlbrace\ ]{align}
     \begin{bmatrix}
        \vI &\\
        & \frac{r^2}{\lambda}\vI-r^2\vA
        \vA^\top
     \end{bmatrix}\begin{bmatrix}
       \zeta\\
       \vs
     \end{bmatrix}\in\begin{bmatrix}
                                    \vI & r\vA^\top \\
                                    -r\vA & \frac{r^2}{\lambda}\vI-r^2\vA\vA^\top+r\partial h^*
                                  \end{bmatrix}&\begin{bmatrix}
                                                 \vx \\
                                                 \vs
                                               \end{bmatrix}, \label{pb3:4a}\\
    %\zeta&\in \vx-r\nabla f(\vx)-r\partial g(\vx)
    (\vI+r\partial g)^{-1} (2\vx-\zeta-r\nabla f(\vx))+\zeta-\vx =&\ \zeta .\label{pb3:4b}
    \end{empheq}
\end{subequations}

Plug $(\vx^{k+1}, \vs^{k+1},\zeta^{k+1})$ into the right-hand side of~\eqref{pb3:sys} and let the rest $\zeta, \vx,\vs$ on the left-hand side of~\eqref{pb3:sys} be the variables in $k$th iteration $(\vx^{k}, \vs^{k},\zeta^{k})$. Then, we can solve for the next-step variables of the system~\eqref{pb3:sys} in the $\vs\mhyphen\vx\mhyphen\zeta$ order by Gaussian elimination. The base algorithm follows
\begin{subequations}\label{algo}
 \begin{empheq}[left=\empheqlfloor\ ]{align}
  \vs^{k+1}&=(\vI+\frac{\lambda}{r}\partial h^*)^{-1}(\frac{\lambda}{r}\vA\zeta^k+(\vI-\lambda \vA\vA^\top)\vs^k),\label{iter:a}\\
  \vx^{k+1}&=\zeta^k-r\vA^\top \vs^{k+1},\label{iter:b}\\
  \zeta^{k+1}&=(\vI+r\partial g)^{-1}(\vx^{k+1}-r\vA^\top \vs ^{k+1}-r\nabla f(\vx^{k+1}))-\vx^{k+1}+\zeta^k,\label{iter:c}
  \end{empheq}
  \end{subequations}
  where $r$ and $\lambda/r$ are primal and dual stepsizes, respectively. 
  
  In Section~\ref{sec:conv}, we will show the convergence of the algorithm when $r<\frac{4\theta-3}{2\theta-1}\frac{2}{L}$ and $\lambda\leq\frac{1}{\theta \sigma^2}$ with any $\theta\in(3/4, 1].$ Notice that, when $\theta=1,$ the conditions on $r$ and $\lambda$ is reduced to  $r<\frac{2}{L}$ and $\lambda\leq\frac{1}{\sigma^2}$. As listed in Table~\ref{table1}, this condition is the sufficient condition on stepizes for PDFP shown in~\cite{chen2016primal} and PD3O shown in~\cite{yan2018new}, which is apparently weaker than that for Condat-Vu shown in~\cite{condat2013primal,vu2013splitting} and AFBA shown in~\cite{latafat2017asymmetric}. This general form of the upper bound gives a larger range of $\lambda$ and accordingly narrows the range of values of $r$.
  
  However, the compromise on $r$ will disappear when $f$ is linear ($L=0$). In this case, $r$ can take any value in $(0,\infty)$ and consequently, $\lambda$ can take value as large as $\frac{4}{3\sigma^2}$. We will later show in Section~\ref{connect:cp} that the base algorithm recovers Chambolle-Pock applied to the dual problem~\eqref{pb2: dual} with $f=0$. The range on $\lambda$ is significantly better than the  previously used for Chambolle-Pock, which is $\lambda\leq\frac{1}{\sigma^2}.$ Therefore, the first question is answered.
  
  Next, we will discuss the connection of the base algorithm with AFBA, PD3O ,Chambolle-Pock and PAPC(PDFP$^2$O). We say an algorithm is equivalent to another one if and only if we can find an one-one map between the sequences generated by two algorithms with appropriate initialization.%The general upper bound on $r$ and $\lambda$ for Algorithm~\ref{algo} also contributes to improving the convergence result for these two algorithms.
 \subsection{Connection with AFBA}\label{connect:afba}
  AFBA is a scheme for general monotone inclusion problem of three-operator splitting. The special form~\cite[Algorithm $5$]{latafat2017asymmetric} used to solve the problem~\eqref{pb2} is conducted as
  \begin{subequations}\label{afba}
 \begin{empheq}[left=\empheqlfloor\ ]{align}
  \vs_1^{k+1}&=(\vI+\frac{\lambda}{r}\partial h^*)^{-1}(\vs_1^k+\frac{\lambda}{r}\vA\overbar{\vx}_1^k),\label{afba:a}\\
  \vx_1^{k+1}&=\overbar{\vx}_1^k-r\vA^\top (\vs_1^{k+1}-\vs_1^k),\label{afba:b}\\
  \overbar{\vx}_1^{k+1}&=(\vI+r\partial g)^{-1}(\vx_1^{k+1}-r\vA^\top \vs_1^{k+1}-r\nabla f(\vx_1^{k+1}))\label{afba:c}.
  \end{empheq}
  \end{subequations}
  %where $\eta_{\text{p}}=r$ and $\eta_{\text{d}}=\frac{\lambda}{r}.$
  We will show that AFBA is equivalent to the base algorithm if the relation 
 \begin{equation}\label{base_afba}
     \begin{bmatrix}
            \vs_1^k\\
            \vx_1^k\\
            \overbar{\vx}_1^k
     \end{bmatrix}=\begin{bmatrix}
            \vs^k\\
            \vx^k\\
            \zeta^k-r\vA^\top\vs^k
     \end{bmatrix},
          \begin{bmatrix}
            \vs^k\\
            \vx^k\\
            \zeta^k
     \end{bmatrix}=\begin{bmatrix}
            \vs_1^k\\
            \vx_1^k\\
            \overbar{\vx}_1^k+r\vA^\top\vs_1^k
     \end{bmatrix}
 \end{equation}
 holds.

 Define $\zeta^{k+1}=\overbar{\vx}_1^{k+1}+r\vA^\top \vs_1^{k+1}$, then~\eqref{afba:a} and~\eqref{afba:b} are reformulated by canceling $\overbar{\vx}$ as
 $$
 \begin{aligned}
 \vs_1^{k+1}&=(\vI+\frac{\lambda}{r}\partial h^*)^{-1}(\vs_1^k+\frac{\lambda}{r}\overbar{\vx}_1^k)=(\vI+\frac{\lambda}{r}\partial h^*)^{-1}(\frac{\lambda}{r}\vA\zeta^k+(\vI-\lambda \vA\vA^\top)\vs_1^k),\\
 \vx_1^{k+1}&=\zeta^k-r\vA^\top \vs_1^k-r\vA^\top(\vs_1^{k+1}-\vs_1^k)=\zeta^k-r\vA^\top\vs_1^{k+1}.
 \end{aligned}$$
  The update in~\eqref{afba:c} gives the update of $\zeta$ as
\begin{equation*}
    \begin{aligned}
    \zeta^{k+1}&=\overbar{\vx}_1^{k+1}+r\vA^\top\vs_1^{k+1}\\
    &=\overbar{\vx}_1^{k+1}+\zeta^k-\vx_1^{k+1}\\
    &=(\vI+r\partial g)^{-1}(\vx_1^{k+1}-r\vA^\top \vs_1 ^{k+1}-r\nabla f(\vx_1^{k+1}))+\zeta^k-\vx_1^{k+1}.
    \end{aligned}
\end{equation*}

Hence, the sequence $\{(\vs_1^k, \vx_1^k)\}$ generated by AFBA coincides with the sequence $\{(\vs^k,\vx^k)\}$ generated by the base algorithm with the same initialization. The reverse direction can be verified by~\eqref{base_afba} in the same way, hence the base algorithm is equivalent to AFBA applied to the problem~\eqref{pb2}.
  
 From~\cite[Proposition $5.3.$]{latafat2017asymmetric}, $(\vx^k,\vs^k)$ generated by AFBA will converge to a saddle point solution $(\vx^*, \vs^*)$ when $\frac{\lambda\sigma^2}{2}+\frac{\sqrt{\lambda}\sigma}{2}+\frac{rL}{2}\leq 1.$ The theoretical result of the base algorithm improves the condition to
 $r\leq\frac{4\theta-3}{2\theta-1}\frac{2}{L}$ and $\lambda\leq\frac{1}{\theta \sigma^2}$ for any  $\theta\in(\frac{3}{4},1].$ Furthermore, when $\theta=1,$ the relaxed condition coincides with the one required for PDFP and PD3O.
  
 \subsection{Connection with PD3O} \label{connect:pd3o}
Applying PD3O to the dual problem~\eqref{pb2: dual}, we get the following iteration,
\begin{subequations}\label{pd3o}\small
 \begin{empheq}[left=\empheqlfloor\ ]{align}
  \vs_2^{k+1}&=(\vI+\frac{\lambda}{r}\partial h^*)^{-1}(\vz_2^k),\label{pd3o:a}\\
  \vx_2^{k+1}&=(\vI+r\partial g)^{-1}((\vI-\lambda\vA^\top\vA)\vx_2^k-r\nabla f(\vx_2^k)-r\vA^\top(2\vs_2^{k+1}-\vz_2^k)),\label{pd3o:b}\\
  \vz_2^{k+1}&=\vs_2^{k+1}+\frac{\lambda}{r}\vA\vx_2^{k+1}.\label{pd3o:c}
  \end{empheq}
\end{subequations}
We can easily verify that the above iteration is equivalent to AFBA if 
  \begin{equation}\label{afba_pd3o}
      \begin{bmatrix}
             \vs_2^k\\
             \vx_2^k\\
             \vz_2^k
      \end{bmatrix}=\begin{bmatrix}
             \vs_1^k\\
             \overbar{\vx}_1^k\\
             \vs_1^k+\frac{\lambda}{r}\vA\overbar{\vx}_1^k
      \end{bmatrix},       \begin{bmatrix}
             \vs_1^k\\
             \vx_1^k\\
             \overbar{\vx}_1^k
      \end{bmatrix}=\begin{bmatrix}
             \vs_2^k\\
             \vx^{k-1}_2-r\vA^\top(\vs_2^k-\vs_2^{k-1})\\
             \vx_2^k
      \end{bmatrix}.
  \end{equation}
  
  Combining~\eqref{pd3o:a} and~\eqref{pd3o:c}, we have
  $$\vs_2^{k+1}=(\vI+\frac{\lambda}{r}\partial h^*)^{-1}(\vs_2^k+\frac{\lambda}{r}\vA\vx_2^k).$$ The equation
  ~\eqref{pd3o:b} is reformulated by cancelling $\vz_2$ as
  \begin{equation*}
      \begin{aligned}
      \vx_2^{k+1}&=(\vI+r\partial g)^{-1}((\vI-\lambda\vA^\top\vA)\vx_2^k-r\nabla f(\vx_2^k)-r\vA^\top(2\vs_2^{k+1}-\vs_2^k-\frac{\lambda}{r}\vA\vx_2^k))\\
      &= (\vI+r\partial g)^{-1}(\vx_2^k-r\nabla f(\vx_2^k)-r\vA^\top(2\vs_2^{k+1}-\vs_2^k)).
      \end{aligned}
  \end{equation*}
  By defining $\vx_1^{k+1}=\vx_2^k-r\vA^\top(\vs^{k+1}_2-\vs_2^k)$, we observe that the sequence $\{(\vs_2^k,\vx_2^k)\}$ generated by PD3O coincides with the sequence $\{(\vs_1^k,\overbar{\vx}_1^k\}$ with the same initialization. The reverse direction also holds from~\eqref{afba_pd3o}.
  
  Due to the equivalence between AFBA and the base algorithm, PD3O is also equivalent to the base algorithm following the sequence relation
    \begin{equation}\label{base_pd3o}
      \begin{bmatrix}
             \vs_2^k\\
             \vx_2^k\\
             \vz_2^k
      \end{bmatrix}=\begin{bmatrix}
             \vs^k\\
             \zeta^k-r\vA^\top\vs^k\\
             (\vI-\lambda\vA\vA^\top)\vs^k+\frac{\lambda}{r}\vA\zeta^k
      \end{bmatrix},       \begin{bmatrix}
             \vs^k\\
             \vx^k\\
             \zeta^k
      \end{bmatrix}=\begin{bmatrix}
             \vs_2^k\\
             \vx^{k-1}_2-r\vA^\top(\vs_2^k-\vs_2^{k-1})\\
             \vx_2^k+r\vA^\top\vs_2^k
      \end{bmatrix}.
  \end{equation}

 \subsection{Connection with Chambolle-Pock}\label{connect:cp}
When the smooth function $f$ vanishes, the dual problem~\eqref{pb2: dual} boils down to
\begin{equation}\label{pb2: dual2}
    \Min_{
    \vs\in\RR^m}\ g^*(-\vA^\top\vs)+h^*(\vs).
\end{equation}
Applying Chambolle-Pock to it, we get
\iffalse\begin{equation}
                            \begin{bmatrix}
                                   \vI+ r\partial g & 2r\vA^\top \\
                                    0 & \frac{r^2}{\lambda}\vI+r\partial h^*
                                  \end{bmatrix}\begin{bmatrix}
                                                 \vx \\
                                                 \vs
                                               \end{bmatrix}\ni\begin{bmatrix}
                                                      \vI& r\vA^\top\\
                                r
                                \vA& \frac{r^2}{\lambda}\vI
                                               \end{bmatrix}\begin{bmatrix}
                                                      \vx\\
                                                      \vs
                                               \end{bmatrix}.
\end{equation}
\fi
%Plug $(\vx^{k+1}, \vs^{k+1})$ and $(\vs^{k},\vx^k)$ into the left-hand side and the right-hand side of the above form, respectively, then we derive Chambolle-Pock in $\vs\mhyphen\vx$ order as
 \begin{subequations}\label{cp}
 \begin{empheq}[left=\empheqlfloor\ ]{align}
  \vs_3^{k+1}&=(\vI+\frac{\lambda}{r}\partial h^*)^{-1}(\vs_3^k+\frac{\lambda}{r}\vA\vx_3^k),\label{cp:a}\\
  \vx_3^{k+1}&=(\vI+r\partial g)^{-1}(\vx_3^{k}-r\vA^\top(2\vs_3^{k+1}-\vs_3^k)).\label{cp:b}
  \end{empheq}
  \end{subequations}

The relation between Chambolle-Pock and the base algorithm can be observed either  from the relation between PD3O and Chambolle-Pock or from that the  sequence $\{(\vs_3^k, \vx_3^k)\}$ generated by Chambolle-Pock coincides with the sequence $\{(\vs^k,\zeta^k-r\vA^\top\vs^k)\}$ generated by the base algorithm with appropriate initialization. 

The sequence relation is
    \begin{equation}\label{base_cp}
      \begin{bmatrix}
             \vs_3^k\\
             \vx_3^k\\
      \end{bmatrix}=\begin{bmatrix}
             \vs^k\\
             \zeta^k-r\vA^\top\vs^k\\
      \end{bmatrix},       \begin{bmatrix}
             \vs^k\\
             \vx^k\\
             \zeta^k
      \end{bmatrix}=\begin{bmatrix}
             \vs_3^k\\
             \vx^{k-1}_3-r\vA^\top(\vs_3^k-\vs_3^{k-1})\\
             \vx_3^k+r\vA^\top\vs_3^k
      \end{bmatrix}.
  \end{equation}

In this case, the relaxed condition on stepsizes of the base algorithm is $r>0$ and $\lambda<\frac{4}{3\sigma^{2}},$ which extends the choice of $\lambda$ to a larger range for Chambolle-Pock. This bound is shown to be tight in Section~\ref{tight}. 
%(\commyl{Mention the tightness of the bound in~\cite{he2020optimally} if the work is cited in introduction})

\subsection{Connection with PAPC(PDFP$^2$O)}\label{connect:papc}
Although in Table~\ref{table1}, AFBA and PD3O are reduced to PAPC when $g=0$, and we have shown the equivalence between the base algorithm and them,  we now directly derive it for completeness.

With $g=0$, the iteration of the base algorithm reduces to
\begin{subequations}\label{algo_g=0}
 \begin{empheq}[left=\empheqlfloor\ ]{align}
  \vs^{k+1}&=(\vI+\frac{\lambda}{r}\partial h^*)^{-1}(\frac{\lambda}{r}\vA\zeta^k+(\vI-\lambda \vA\vA^\top)\vs^k),\label{iter:a2}\\
  \vx^{k+1}&=\zeta^k-r\vA^\top \vs^{k+1},\label{iter:b2}\\
  \zeta^{k+1}&=\vx^{k+1}-r\nabla f(\vx^{k+1}).\label{iter:c2}
  \end{empheq}
  \end{subequations}
Let $\vs_4^k = \vs^k$ and $\vx_4^k=\vx^k$. After canceling $\zeta$, we get
\begin{subequations}\label{papc}
 \begin{empheq}[left=\empheqlfloor\ ]{align}
  \vs_4^{k+1}&=(\vI+\frac{\lambda}{r}\partial h^*)^{-1}(\frac{\lambda}{r}\vA(\vx_4^k-r\nabla f(\vx_4^{k}))+(\vI-\lambda \vA\vA^\top)\vs_4^k),\label{papc:a}\\
  \vx_4^{k+1}&=\vx_4^k-r\nabla f(\vx_4^k)-r\vA^\top\vs_4^{k+1}.\label{papc:b}
  \end{empheq}
  \end{subequations}
  With appropriation initialization, i.e., $\zeta^0=\vx^0-r\nabla f(\vx^0),$ the above iteration is exactly PAPC applied to the problem~\eqref{pb1}. However, the relaxed condition, $r<\frac{4\theta-3}{2\theta-1}\frac{2}{L}$ and $\lambda<\frac{1}{\theta\sigma^2(\vA)}$ for $\theta\in(3/4,1]$ cannot achieve the tight bound shown in~\cite{li2021new}. One of the possible reason is that we cannot express $\vA^\top\vs$ as a combination of $\vx$ and $\nabla f(\vx)$ like~\eqref{papc:b} when nontrivial $g$ exists. This will be left for future research.

\subsection{Relation Diagram}
The relation between aforementioned algorithms is visualized via the following diagram where the algorithms in boxed are applied to the dual problem~\eqref{pb2: dual} with $\eta_{\text{p}}=\lambda/r$ and $\eta_{\text{d}}=r$.
\begin{equation*}
\begin{tikzcd}
\textbf{PAPC(PDFP$^2$O)\eqref{papc}}                                                                                                       &  &                                              &  &                       \\
\textbf{The Base Algorithm\eqref{algo}} \arrow[u, "g=0"] \arrow[d, "f=0"', "\eqref{base_cp}"] \arrow[rr, "\eqref{base_afba}"', Leftrightarrow] \arrow[rrrr, "\eqref{base_pd3o}", Leftrightarrow, bend left] &  & \textbf{AFBA\eqref{afba}} \arrow[rr, "\eqref{afba_pd3o}"', Leftrightarrow] &  & \boxed{\textbf{PD3O\eqref{pd3o}}} \\
\boxed{\textbf{Chambolle-Pock\eqref{cp}}}                                                                                                          &  &                                              &  &                      
\end{tikzcd}
\end{equation*}

The theoretical analysis of the base algorithm in the next section provides a unified proof of the convergence for the above algorithms under the relaxed conditions on $r$ and $\lambda.$ An example is given in Section~\ref{tight} to show the necessary condition of $r$ and $\lambda$.

\section{Convergence Analysis }\label{sec:conv}
\subsection{Convergence under a relaxed condition}
We first define the following auxiliary variables associated with the proximal mapping of $h^*$ and $g$
\begin{subequations}
    \begin{align}
    \vy^{k+1}&=(\vI+r\partial g)^{-1}(\vx^{k+1}-r\vA^\top \vs ^{k+1}-r\nabla f(\vx^{k+1})),\\
    \vq_g^{k+1}&=\frac{1}{r}\vx^{k+1}-\vA^\top \vs^{k+1}-\nabla f(\vx^{k+1})-\frac{1}{r} \vy^{k+1},\\
    \vq_h^{k+1}&=\vA\zeta^k+\frac{r}{\lambda}(\vI-\lambda \vA\vA^\top)\vs^k-\frac{r}{\lambda}\vs^{k+1}.
    \end{align}
\end{subequations}

With the existence of the saddle-point solution $(\vx^*, \vs^*),$ we will see the quadruple $(\zeta^{k+1}, \vy^{k+1}, \vq_g^{k+1}, \vq_h^{k+1})$ converges to
\begin{subequations}
    \begin{align}
    \zeta^*& =\vx^*+r\vA^\top \vs^*,\label{zeta_star}\\
    \vy^*&=\vx^*,\label{y_star}\\
    \vq_g^*&=-\vA^\top\vs^*-\nabla f(\vx^*),\\
    \vq_h^*&=\vA\zeta^*-r\vA\vA^\top \vs^*=\vA\vx^*.
    \end{align}
\end{subequations}

Though the relation between the base algorithm, AFBA and PD3O asserts the fixed point of iteration~\eqref{algo} solves the problem~\eqref{pb2}, we restate it in the following lemma for completeness.
\begin{lemma}[Optimality]\label{lem}
Let $(\vs^\star,\vx^\star,\zeta^\star)$ be the fixed point of the iteration~\eqref{algo}, then $(\vs^\star,\vx^\star)$ is the solution of the problem~\eqref{pb2} with $\zeta^\star=\vx^\star+r\vA^\top\vs^\star.$
\end{lemma}
\begin{proof}
Iteration~\eqref{iter:b} gives the relation $\zeta^\star =\vx^\star+r\vA^\top\vs^\star.$
From~\eqref{iter:a} and~\eqref{iter:c}, we have
\begin{equation*}
    \begin{aligned}
    \vs^\star+\frac{\lambda}{r}\partial h^*(\vs^\star)&=\frac{\lambda}{r}\vA\zeta^\star+(\vI-\lambda\vA\vA^\top)\vs^\star,\\
    \vx^\star+r\partial g(\vx^\star)&=\vx^\star-r\vA^\top\vs^\star-r\nabla f(\vx^\star).
    \end{aligned}
\end{equation*}

Canceling $\zeta^*$, we have
\begin{equation*}
    \begin{bmatrix}
           \partial h^*& -\vA\\
           \vA^\top&\partial g+\nabla f
    \end{bmatrix}\begin{bmatrix}
           \vs^\star\\
           \vx^\star
    \end{bmatrix}=\begin{bmatrix}
           0\\
           0
    \end{bmatrix},
\end{equation*}
which is the optimal condition of the problem~\eqref{pb2}.\qed
\end{proof}

 We now use $(\vs^*,\vx^*,\zeta^*)$ to denote an arbitrary fixed point of the base algorithm.

\begin{lemma}[Fundamental equality]\label{lem1}
Let the sequence $\{(\vs^k, \vx^k,\zeta^k)\}$ be generated by the iteration~\eqref{algo}, we have the following two equalities hold:
    \begin{align}
    &\ \dotp{\vy^k-\vy^*,\vq_g^k-\vq_g^*}+\dotp{\vs^{k+1}-\vs^k,\vq_h^{k+1}-\vq_h^*}\nonumber\\
    =&\ \frac{1}{r}\dotp{\vx^{k+1}-\vx^*,\vx^k-\vx^{k+1}}+T_1+\frac{r}{\lambda}\dotp{\vs^{k+1}-\vs^*,\vs^k-\vs^{k+1}}+T_2\label{lem:eq1}
    \end{align}
    and
    \begin{align}
    &\ \dotp{\vy^{k+1}-\vy^*,\vq_g^{k+1}-\vq_g^*}+\dotp{\vs^{k+1}-\vs^k,\vq_h^{k+1}-\vq_h^*}\nonumber \\
    =&\ \frac{1}{r}\dotp{\zeta^{k+1}-\zeta^*,\zeta^k-\zeta^{k+1}}+\frac{r}{\lambda}\dotp{\vs^{k+1}-\vs^*,\vs^k-\vs^{k+1}}\nonumber\\
    &\ +r\dotp{\vA^\top(\vs^{k+1}-\vs^{*}),\vA^\top(\vs^{k+1}-\vs^k)}+T_3,\label{lem:eq2}
    \end{align}
where
\begin{equation*}
    \begin{aligned}
    &T_1 \coloneqq \dotp{\vA^\top(\vs^{k+1}-\vs^k),\vx^k-\vx^{k+1}},\\
    &T_2 \coloneqq \dotp{\vy^*-\vy^k,\nabla f(\vx^k)-\nabla f(\vx^*)},\\
    &T_3 \coloneqq \dotp{\vy^*-\vy^{k+1},\nabla f(\vx^{k+1})-\nabla f(\vx^*)}.
    \end{aligned}
\end{equation*}
\end{lemma}
\begin{proof}
Plugging the iteration~\eqref{iter:b} into $\vq_h^{k+1}$ to cancel $\zeta^k,$ we have
\begin{align}
 \vq_h^{k+1}=&\ \vA\vx^{k+1}+r\vA\vA^\top\vs^{k+1}+\frac{r}{\lambda}(\vI-\lambda\vA\vA^\top)\vs^k-\frac{r}{\lambda}\vs^{k+1}\nonumber\\
  =&\ \vA\vx^{k+1}+\frac{r}{\lambda}(\vI-\lambda\vA\vA^\top)(\vs^k-\vs^{k+1}).
\end{align}

For the equality~\eqref{lem:eq1}, we have
    \begin{align}
    &\ \dotp{\vy^k-\vy^{*}, \vq_g^{k}-\vq_g^{*}}+\dotp{\vs^{k+1}-\vs^{*}, \vq_h^{k+1}-\vq_h^{*}}\nonumber\\
  =&\ \underbrace{\frac{1}{r}\dotp{\vy^k-\vy^*,\vx^k-\vx^*}-\frac{1}{r}\|\vy^k-\vy^*\|^2}_{=\frac{1}{r}\dotp{\vy^k-\vy^*, \vx^k-\vy^k}}-\dotp{\vy^k-\vy^*,\vA^\top(\vs^k-\vs^*)}+T_2\nonumber\\
  &\ +\dotp{\vs^{k+1}-\vs^*, \vA(\vx^{k+1}-\vx^*)}+\frac{r}{\lambda}\dotp{\vs^{k+1}-\vs^*,(\vI-\lambda \vA\vA^\top)(\vs^k-\vs^{k+1})}\nonumber\\
  =&\ \frac{1}{r}\dotp{\vy^k-\vy^*,\vx^k-\vx^{k+1}}\underbrace{-\dotp{\vy^k-\vy^*, \vA^\top(\vs^{k+1}-\vs^k)}-\dotp{\vy^k-\vy^*,\vA^\top(\vs^k-\vs^*)}}_{=-\dotp{\vy^k-\vy^*,\vA^\top(\vs^{k+1}-\vs^*)}}\nonumber\\
  &\ +T_2+\dotp{\vs^{k+1}-\vs^*, \vA(\vx^{k+1}-\vx^*)}+\frac{r}{\lambda}\dotp{\vs^{k+1}-\vs^*,(\vI-\lambda \vA\vA^\top)(\vs^k-\vs^{k+1})}\nonumber\\
    =&\ \frac{1}{r}\dotp{\vy^k-\vy^*,\vx^k-\vx^{k+1}}+\dotp{\vx^{k+1}-\vy^k, \vA^\top(\vs^{k+1}-\vs^*)}+T_2\nonumber\\
  &\ +\frac{r}{\lambda}\dotp{\vs^{k+1}-\vs^*,(\vI-\lambda \vA\vA^\top)(\vs^k-\vs^{k+1})},
    \end{align}
where the second equality uses
\begin{align}
    \vx^k-\vx^{k+1}=&\ \vx^k-\zeta^k+r\vA^\top\vs^{k+1}\nonumber\\ 
    =&\ \vx^{k}-\vy^{k}+r\vA^\top(\vs^{k+1}-\vs^k),\label{lem:eq3}
\end{align}
which is derived from the iteration~\eqref{iter:b} and~\eqref{iter:c}.

Note that from the equality~\eqref{lem:eq3}, we also have
\begin{align}\label{lem:eq4}
    \vx^{k+1}-\vy^k=-r\vA^\top(\vs^{k+1}-\vs^k).
\end{align}
Hence
\begin{align}
    &\ \dotp{\vy^k-\vy^{*}, \vq_g^{k}-\vq_g^{*}}+\dotp{\vs^{k+1}-\vs^{*}, \vq_h^{k+1}-\vq_h^{*}}\nonumber\\
    =&\ \frac{1}{r}\dotp{\vy^k-\vy^*,\vx^k-\vx^{k+1}}-r\dotp{\vs^{k+1}-\vs^k, \vA\vA^\top(\vs^{k+1}-\vs^*)}+T_2\nonumber\\
  &\ +\frac{r}{\lambda}\dotp{\vs^{k+1}-\vs^*,(\vI-\lambda \vA\vA^\top)(\vs^k-\vs^{k+1})}\nonumber \\
  =&\ \frac{1}{r}\dotp{\vy^k-\vy^*,\vx^k-\vx^{k+1}}+\frac{r}{\lambda}\dotp{\vs^{k+1}-\vs^*,\vs^k-\vs^{k+1}}+T_2\nonumber\\
  =&\frac{1}{r}\dotp{\vx^{k+1}-\vx^*,\vx^k-\vx^{k+1}}+T_1+\frac{r}{\lambda}\dotp{\vs^{k+1}-\vs^*,\vs^k-\vs^{k+1}}+T_2
\end{align}
and the equality~\eqref{lem:eq1} is derived.

For the equality~\eqref{lem:eq2}, we have
\begin{align}
     &\ \dotp{\vy^{k+1}-\vy^{*}, \vq_g^{k+1}-\vq_g^{*}}+\dotp{\vs^{k+1}-\vs^{*}, \vq_h^{k+1}-\vq_h^{*}}\nonumber\\
     =&\ \frac{1}{r}\dotp{\vy^{k+1}-\vy^*,\vx^{k+1}-\vy^{k+1}}-\dotp{\vy^{k+1}-\vy^*,\vA^\top(\vs^{k+1}-\vs^*)}+T_3\nonumber\\
     &\ +\dotp{\vs^{k+1}-\vs^*,\vA(\zeta^k-\zeta^*)}+\frac{r}{\lambda}\dotp{\vs^{k+1}-\vs^*,(\vI-\lambda\vA\vA^\top)\vs^k-\vs^{k+1}+\lambda\vA\vA^\top\vs^*}\nonumber\\
    =&\ \frac{1}{r}\dotp{\vy^{k+1}-\vy^*,\zeta^{k}-\zeta^{k+1}}-\dotp{\vy^{k+1}-\vy^*,\vA^\top(\vs^{k+1}-\vs^*)}+T_3\nonumber\\
     &\ +\underbrace{\dotp{\vs^{k+1}-\vs^*,\vA(\zeta^k-\zeta^*)}+r\dotp{\vs^{k+1}-\vs^*,\vA\vA^\top
     (\vs^*-\vs^k)}}_{=\dotp{\vA^\top(\vs^{k+1}-\vs^*), \vy^k-\vy^*}}\nonumber\\
     &\ +\frac{r}{\lambda}\dotp{\vs^{k+1}-\vs^*,\vs^k-\vs^{k+1}}\nonumber \\
     =&\ \frac{1}{r}\dotp{\vy^{k+1}-\vy^*,\zeta^{k}-\zeta^{k+1}}+\dotp{\vy^{k}-\vy^{k+1},\vA^\top(\vs^{k+1}-\vs^*)}+T_3\nonumber\\
    &\ +\frac{r}{\lambda}\dotp{\vs^{k+1}-\vs^*,\vs^k-\vs^{k+1}}.
\end{align}
Note that from the equality~\eqref{lem:eq3}, we also have
\begin{align}\vy^{k+1}=\zeta^{k+1}-r\vA^\top\vs^{k+1}\label{lem:eq5}.\end{align}
Combining it with~\eqref{zeta_star} and~\eqref{y_star}, we have
\begin{align}
    \vy^{k+1}-\vy^*=\zeta^{k+1}-\zeta^*-r\vA^\top(\vs^{k+1}-\vs^*),
\end{align}
then
\begin{align}
    &\ \dotp{\vy^{k+1}-\vy^{*}, \vq_g^{k+1}-\vq_g^{*}}+\dotp{\vs^{k+1}-\vs^{*}, \vq_h^{k+1}-\vq_h^{*}}\nonumber\\
    =&\ \frac{1}{r}\dotp{\zeta^{k+1}-\zeta^*,\zeta^k-\zeta^{k+1}}+\frac{r}{\lambda}\dotp{\vs^{k+1}-\vs^*,\vs^k-\vs^{k+1}}+T_3\nonumber\\
    &\ +\dotp{\vy^{k}-\vy^{k+1},\vA^\top(\vs^{k+1}-\vs^*)}-\dotp{\zeta^k-\zeta^{k+1},\vA^\top(\vs^{k+1}-\vs^*)}\nonumber\\
    =&\ \frac{1}{r}\dotp{\zeta^{k+1}-\zeta^*,\zeta^k-\zeta^{k+1}}+\frac{r}{\lambda}\dotp{\vs^{k+1}-\vs^*,\vs^k-\vs^{k+1}}+T_3\nonumber\\
    &\ +r\dotp{\vA^\top(\vs^{k+1}-\vs^k),\vA^\top(\vs^{k+1}-\vs^*)},
\end{align}
where the last equality uses~\eqref{lem:eq5} and the equality~\eqref{lem:eq2} is derived.  \qed
\end{proof}

The next lemma characterizes the upper bound for $T_1, T_2$ and $T_3$.  

\begin{lemma}\label{lem2}
For any $\theta\in(3/4, 1]$ and $\theta\lambda\sigma^2\leq 1$, we have the following inequalities hold.
\begin{align}\label{lem2:ineq1}
    T_1\leq&\ \frac{r}{4\lambda}\|\vs^k-\vs^{k-1}\|^2_{\vM}-\frac{r}{4\lambda}\|\vs^{k+1}-\vs^k\|^2_{\vM}+\frac{1}{4}(1-\theta)r\|\vA^\top(\vs^k-\vs^{k-1})\|^2\nonumber\\
    &\ +(\frac{1}{2}-\frac{1}{4}\theta)r\|\vA^\top(\vs^{k+1}-\vs^k)\|^2+(\frac{5}{4}-\theta)\frac{1}{r}\|\vx^{k}-\vx^{k+1}\|^2,
\end{align}
\begin{align}\label{lem2:ineq2}
    T_2\leq&\ \frac{L}{4}\|\zeta^{k-1}-\zeta^k\|^2\nonumber\\
    =&\ \frac{L}{4}\|\vx^k-\vx^{k+1}\|^2+\frac{r^2L}{4}\|\vA^\top(\vs^k-\vs^{k+1})\|^2-\frac{rL}{2}T_1,
\end{align}
\begin{align}\label{lem2:ineq3}
    T_3\leq&\ \frac{L}{4}\|\zeta^k-\zeta^{k+1}\|^2,
\end{align}
where  $T_1, T_2, T_3$ are defined in Lemma~\ref{lem1} and $\vM\coloneqq \vI-\theta\lambda\vA\vA^\top\succcurlyeq\vzero.$
\end{lemma}
\begin{proof}
We start from the upper bound of $T_3.$ Since $\vy^*=\vx^*$,
\begin{align}
    T_3=&\ \dotp{\vx^{k+1}-\vy^{k+1},\nabla f(\vx^{k+1})-\nabla f(\vx^*)}-\dotp{\vx^{k+1}-\vx^{*},\nabla f(\vx^{k+1})-\nabla f(\vx^*)}\nonumber\\
    \leq&\ \dotp{\vx^{k+1}-\vy^{k+1},\nabla f(\vx^{k+1})-\nabla f(\vx^*)}-\frac{1}{L}\|\nabla f(\vx^{k+1})-\nabla f(\vx^*)\|^2\nonumber\\
    \leq&\ \frac{L}{4}\|\vx^{k+1}-\vy^{k+1}\|^2,
\end{align}
where the second inequality uses the equivalent form of Lipschitz continuous $\nabla f$ from~\cite[Theorem 2.1.5]{nesterov2018lectures}  and the last one uses Cauchy's inequality.

To upper bound $T_2,$ by the same argument as above, we have
\begin{align}
    T_2\leq&\ \frac{L}{4}\|\vx^k-\vy^k\|^2\nonumber\\
    =&\ \frac{L}{4}\|\vx^k-\vx^{k+1}-r\vA^\top(\vs^{k+1}-\vs^k)\|^2,
\end{align}
where the equality is from~\eqref{lem:eq3} and~\eqref{lem2:ineq2} is derived by expanding it.

For $T_1,$ firstly we have
  \begin{align}
  T_1 =&\ \dotp{\vs^{k+1}-\vs^k, \frac{r}{\lambda}(\vI-\lambda \vA\vA^\top)(\vs^k-\vs^{k-1})-\frac{r}{\lambda}(\vI-\lambda \vA\vA^\top)(\vs^{k+1}-\vs^k)}\nonumber\\
  &\ -\dotp{\vs^{k+1}-\vs^k,\vq_h^{k+1}-\vq_h^{k}}\nonumber\\
  \leq &\ \dotp{\vs^{k+1}-\vs^k, \frac{r}{\lambda}\vM(\vs^k-\vs^{k-1}-(\vs^{k+1}-\vs^k))}\nonumber\\
  &-(1-\theta)r\dotp{\vs^{k+1}-\vs^k,\vA\vA^\top(\vs^k-\vs^{k-1}-(\vs^{k+1}-\vs^k))}\nonumber\\
  =&\ \frac{r}{2\lambda}\|\vs^k-\vs^{k-1}\|^2_{\vM}-\frac{r}{2\lambda}\|\vs^{k+1}-\vs^k\|^2_{\vM}-\frac{r}{2\lambda}\|\vs^{k+1}+\vs^{k-1}-2\vs^{k}\|^2_{\vM}\nonumber\\
  &\ -(1-\theta)r\dotp{\vA^\top(\vs^{k+1}-\vs^k),\vA^\top(\vs^k-\vs^{k-1}-(\vs^{k+1}-\vs^k))}\nonumber\\
  \leq&\ \frac{r}{2\lambda}\|\vs^k-\vs^{k-1}\|^2_{\vM}-\frac{r}{2\lambda}\|\vs^{k+1}-\vs^k\|^2_{\vM}+\frac{1}{2}(1-\theta)r\|\vA^\top(\vs^{k}-\vs^{k-1})\|^2\nonumber\\
   &\ +\frac{3}{2}(1-\theta)r\|\vA^\top(\vs^{k+1}-\vs^k)\|^2\label{lem2:ineq4},
  \end{align}
where the first inequality uses the firm nonexpansiveness of $\prox_{\frac{\lambda}{r}h^*}$ from~\cite[Proposition 4.2]{bauschke2011convex}, the second equality uses
$$\dotp{a-b, c-d}=\frac{1}{2}(\|a-d\|^2-\|a-c\|^2+\|b-c\|^2-\|b-d\|^2)$$
and the last step uses Chauchy's inequality.

On the other hand, we also have
\begin{align}
    T_1=&\ 2\dotp{\frac{1}{\sqrt{2}}\vA(\vs^{k+1}-\vs^k),\frac{1}{\sqrt{2}}(\vx^{k}-\vx^{k+1})}\nonumber\\
    \leq&\ 2\dotp{\sqrt{(\theta-\frac{1}{2})}\vA(\vs^{k+1}-\vs^k),\sqrt{(\frac{5}{2}-2\theta)}(\vx^{k}-\vx^{k+1})}\nonumber\\
    \leq&\ (\theta-\frac{1}{2})\|\vA(\vs^{k+1}-\vs^k)\|^2+(\frac{5}{2}-2\theta)\|\vx^k-\vx^{k+1}\|^2,\label{lem2:ineq5}
\end{align}
where the first inequality holds since for any $\theta\in(\frac{3}{4},1]$
$$(\theta-\frac{1}{2})(\frac{5}{2}-2\theta)-\frac{1}{4}=-2\theta^2+\frac{7}{2}\theta-\frac{3}{2}=2(\theta-\frac{3}{4})(1-\theta)\geq0.$$

Therefore, 
$$T_1\leq \frac{1}{2}\times\eqref{lem2:ineq4}+\frac{1}{2}\times\eqref{lem2:ineq5},$$
which gives the inequality~\eqref{lem2:ineq1}. \qed
\end{proof}

\begin{theorem}\label{thm}
Let the sequence $\{(\vs^k, \vx^k,\zeta^k)\}$ be generated by the base algorithm,  then $\{(\vs^k,\vx^k)\}$ converges to a solution of the problem~\eqref{pb2}, if
$$r<\frac{4\theta-3}{2\theta-1}\frac{2}{L},\ \   \lambda\leq\frac{1}{\theta\sigma^2}$$
for any $\theta\in(3/4,1].$
\end{theorem}
\begin{proof}{\it Descent Inequality.} 
From the assumption on $\lambda,$ there exists $\widetilde{\theta}\in(0,\theta)$ such that
$$\vI-\widetilde{\theta}\lambda\vA\vA^\top\succ\vzero.$$
Define a parameter $\alpha$ depending on $\theta$ and $\widetilde{\theta}$ as
\begin{equation}
    \alpha=\left\{\begin{aligned}
    \frac{\widetilde{\theta}}{1-\theta},&\ \ \frac{3}{4}<\theta<1,\\
    0,&\ \ \theta=1,
    \end{aligned}\right.
\end{equation}
and let $\widetilde{\vM}=\vI-\alpha(1-\theta)\vA\vA^\top\succ\vzero$.

Combining Lemma~\ref{lem1} and Lemma~\ref{lem2}, for any $r\in(0,2/L)$ we have
\begin{align}
    &\ \dotp{\vy^k-\vy^*,\vq_g^k-\vq_g^*}+\dotp{\vs^{k+1}-\vs^k,\vq_h^{k+1}-\vq_h^*}\nonumber\\
    =&\ \frac{1}{r}\dotp{\vx^{k+1}-\vx^*,\vx^k-\vx^{k+1}}+\frac{r}{\lambda}\dotp{\vs^{k+1}-\vs^*,\vs^k-\vs^{k+1}}+T_1+T_2\nonumber\\
    \leq&\ \frac{1}{2r}\|\vx^{k}-\vx^*\|^2-\frac{1}{2r}\|\vx^{k+1}-\vx^*\|^2-\frac{1}{2r}\|\vx^k-\vx^{k+1}\|^2\nonumber\\
    &\ +\frac{r}{2\lambda}\|\vs^{k}-\vs^*\|^2-\frac{r}{2\lambda}\|\vs^{k+1}-\vs^*\|^2-\frac{r}{2\lambda}\|\vs^{k+1}-\vs^k\|^2\nonumber\\
    &\ +\alpha(\frac{1}{2r}-\frac{L}{4})\|\zeta^{k-1}-\zeta^k\|^2+(\frac{L}{4}-\alpha(\frac{1}{2r}-\frac{L}{4}))\|\zeta^{k-1}-\zeta^{k}\|^2\nonumber\\
    &\ +T_1\nonumber\\
    \leq&\ \frac{1}{2r}\|\vx^{k}-\vx^*\|^2-\frac{1}{2r}\|\vx^{k+1}-\vx^*\|^2-\frac{1}{2r}\|\vx^k-\vx^{k+1}\|^2\nonumber\\
    &\ +\frac{r}{2\lambda}\|\vs^{k}-\vs^*\|^2-\frac{r}{2\lambda}\|\vs^{k+1}-\vs^*\|^2-\frac{r}{2\lambda}\|\vs^{k+1}-\vs^k\|^2\nonumber\\
    &\ +\alpha(\frac{1}{2r}-\frac{L}{4})\|\zeta^{k-1}-\zeta^k\|^2+(\frac{L}{4}-\alpha(\frac{1}{2r}-\frac{L}{4}))\|\vx^k-\vx^{k+1}\|^2\nonumber\\
    &\ +(\frac{L}{4}-\alpha(\frac{1}{2r}-\frac{L}{4}))r^2\|\vA^\top(\vs^k-\vs^{k+1})\|^2+(1+\alpha)(1-\frac{rL}{2})T_1.\label{thm:ineq1}
    \end{align}
    
    Note that \begin{align}
        (1+\alpha)(1-\frac{rL}{2})T_1\leq&\ (1+\alpha)(1-\frac{rL}{2})\Big[\frac{r}{4\lambda}\|\vs^k-\vs^{k-1}\|^2_{\vM}-\frac{r}{4\lambda}\|\vs^{k+1}-\vs^k\|^2_{\vM}\nonumber\\
        &\ +\frac{1}{4}(1-\theta)r\|\vA^\top(\vs^k-\vs^{k-1})\|^2\nonumber\\
    &\ +(\frac{1}{2}-\frac{1}{4}\theta)r\|\vA^\top(\vs^{k+1}-\vs^k)\|^2-(\theta-\frac{3}{4})\frac{1}{r}\|\vx^{k}-\vx^{k+1}\|^2\Big]\nonumber\\
    &\ +(1+\alpha)(1-\frac{rL}{2})\frac{1}{2r}\|\vx^{k}-\vx^{k+1}\|^2.\label{thm:ineq2}
    \end{align}
    
    Combine~\eqref{thm:ineq1},~\eqref{thm:ineq2}, and
    $$-\frac{1}{2r}+\frac{L}{4}-\alpha(\frac{1}{2r}-\frac{L}{4})+(1+\alpha)(\frac{1}{2r}-\frac{L}{4})=0,$$
    we get
    \begin{align}\label{thm:ineq3}
        &\ \dotp{\vy^k-\vy^*,\vq_g^k-\vq_g^*}+\dotp{\vs^{k+1}-\vs^k,\vq_h^{k+1}-\vq_h^*}\nonumber\\
        \leq&\ \frac{1}{2r}\|\vx^{k}-\vx^*\|^2-\frac{1}{2r}\|\vx^{k+1}-\vx^*\|^2+\frac{r}{2\lambda}\|\vs^{k}-\vs^*\|^2-\frac{r}{2\lambda}\|\vs^{k+1}-\vs^*\|^2\nonumber\\
    &\ +\alpha(\frac{1}{2r}-\frac{L}{4})\|\zeta^{k-1}-\zeta^k\|^2-\frac{r}{2\lambda}\|\vs^{k+1}-\vs^k\|^2\nonumber\\
    &\ +(\frac{L}{4}-\alpha(\frac{1}{2r}-\frac{L}{4}))r^2\|\vA^\top(\vs^k-\vs^{k+1})\|^2\nonumber\\
    &\ (1+\alpha)(1-\frac{rL}{2})\Big[\frac{r}{4\lambda}\|\vs^k-\vs^{k-1}\|^2_{\vM}-\frac{r}{4\lambda}\|\vs^{k+1}-\vs^k\|^2_{\vM}\nonumber\\
        &\ +\frac{1}{4}(1-\theta)r\|\vA^\top(\vs^k-\vs^{k-1})\|^2-\frac{1}{4}(1-\theta)r\|\vA^\top(\vs^{k+1}-\vs^{k})\|^2\nonumber\\
        &\ +(\frac{3}{4}-\frac{1}{2}\theta)r\|\vA^\top(\vs^{k+1}-\vs^k)\|^2-(\theta-\frac{3}{4})\|\vx^k-\vx^{k+1}\|^2\Big].
    \end{align}
    
The other inequality in Lemma~\ref{lem1} becomes
\begin{align}\label{thm:ineq4}
 &\ \dotp{\vy^{k+1}-\vy^*,\vq_g^{k+1}-\vq_g^*}+\dotp{\vs^{k+1}-\vs^k,\vq_h^{k+1}-\vq_h^*}\nonumber \\
    =&\ \frac{1}{r}\dotp{\zeta^{k+1}-\zeta^*,\zeta^k-\zeta^{k+1}}+\frac{r}{\lambda}\dotp{\vs^{k+1}-\vs^*,\vM(\vs^k-\vs^{k+1})}\nonumber\\
    &\ -(1-\theta)r\frac{r}{\lambda}\dotp{\vA^\top(\vs^{k+1}-\vs^*),\vA^\top(\vs^k-\vs^{k+1})}+T_3\nonumber\\
    \leq&\ \frac{1}{2r}\|\zeta^{k}-\zeta^*\|^2-\frac{1}{2r}\|\zeta^{k+1}-\zeta^*\|^2-\frac{1}{2r}\|\zeta^{k+1}-\zeta^k\|^2\nonumber\\
    &\ +\frac{r}{2\lambda}\|\vs^{k}-\vs^*\|^2_{\vM}-\frac{r}{2\lambda}\|\vs^{k+1}-\vs^*\|^2_{\vM}-\frac{r}{2\lambda}\|\vs^{k+1}-\vs^k\|^2_{\vM}\nonumber\\
    &\ -(1-\theta)\frac{r}{2}\|\vA^\top(\vs^{k}-\vs^*)\|^2+(1-\theta)\frac{r}{2}\|\vA^\top(\vs^{k+1}-\vs^*)\|^2\nonumber\\
    &\ +(1-\theta)\frac{r}{2}\|\vA^\top(\vs^{k+1}-\vs^k)\|^2+\frac{L}{4}\|\zeta^k-\zeta^{k+1}\|^2.
\end{align}

Let $\beta=(1+\alpha)(1-\frac{rL}{2})$ for simplicity. We have 
\begin{equation*}\begin{aligned}
    \Phi^{k+1}=&\ \frac{1}{2r}\|\vx^{k+1}-\vx^*\|^2+\frac{r}{2\lambda}\|\vs^{k+1}-\vs^*\|^2_{\vM+\widetilde{\vM}}+\frac{1}{2r}\|\zeta^{k+1}-\zeta^*\|^2\\
    &\ +\frac{\beta r}{4\lambda}\|\vs^{k+1}-\vs^{k}\|^2_{\vM}+\frac{\beta r(1-\theta)}{4}\|\vs^{k+1}-\vs^k\|^2_{\vA\vA^\top}\\
    &\ +\frac{\alpha}{2r}(1-\frac{rL}{2})\|\zeta^{k}-\zeta^{k+1}\|^2.
    \end{aligned}
\end{equation*}
Considering $\eqref{thm:ineq3}+\alpha\times\eqref{thm:ineq4}$, we have  
\begin{align}
    &\ \dotp{\vy^k-\vy^*,\vq_g^k-\vq_g^*}+(1+\alpha)\dotp{\vs^{k+1}-\vs^k,\vq_h^{k+1}-\vq_h^*}+\alpha\dotp{\vy^{k+1}-\vy^*,\vq_g^{k+1}-\vq_g^*}\nonumber\\
    \leq&\ \Phi^{k}-\Phi^{k+1}-\beta(\theta-\frac{3}{4})\|\vx^k-\vx^{k+1}\|^2-\frac{r}{2\lambda}\|\vs^{k+1}-\vs^k\|^2_{\vM}\nonumber\\
    &\ -\frac{r}{2\lambda}\|\vs^{k+1}-\vs^k\|^2_{\vI-\alpha(1-\theta)\lambda\vA\vA^\top}+(\frac{3}{4}-\frac{1}{2}\theta)\beta r\|\vs^{k+1}-\vs^k\|^2_{\vA\vA^\top}\nonumber\\
    &\ +(\frac{rL}{2}-\alpha(1-\frac{rL}{2}))\frac{r}{2}\|\vs^k-\vs^{k+1}\|_{\vA\vA^\top}^2\nonumber\\
    =&\ \Phi^{k}-\Phi^{k+1}-\beta(\theta-\frac{3}{4})\|\vx^k-\vx^{k+1}\|^2-\frac{r}{2\lambda}\|\vs^{k+1}-\vs^k\|^2_{\vM}\nonumber\\
    &\ -\frac{r}{2\lambda}\|\vs^{k+1}-\vs^k\|^2_{\vN},
\end{align}
where
$$\vN=\vI-\Big[\alpha(1-\theta)+(\frac{3}{2}-\theta)\beta+\frac{rL}{2}-\alpha(1-\frac{rL}{2})\Big]\lambda\vA\vA^\top.$$

Note that if we let 
$$\alpha(1-\theta)+(\frac{3}{2}-\theta)(1+\alpha)(1-\frac{rL}{2})+\frac{rL}{2}-\alpha(1-\frac{rL}{2})< \theta,$$
which is equivalent to 
$$\frac{rL}{2}<\frac{4\theta-3}{2\theta-1},$$
we have $\vN\succ\vM\succcurlyeq\vzero.$

Due to $\widetilde{\vM}\succ\vzero$ and  the firm nonexpansiveness of $\prox_{rg}$ and $\prox_{\frac{\lambda}{r}h^*}$, from~\cite[Proposition 4.2]{bauschke2011convex}, we get the descent inequality
\begin{equation}\Phi^{k+1}\leq\Phi^k-\beta(\theta-\frac{3}{4})\|\vx^{k+1}-\vx^{k}\|^2-\frac{r}{2\lambda}\|\vs^{k+1}-\vs^k\|^2_{\vM+\vN}.
\end{equation} 

\paragraph{Convergence.} Taking the telescopic sum from $k=0$ to $\infty$, we get
\begin{align}
    &\lim_{k\rightarrow\infty}\|\vx^{k+1}-\vx^k\|=0, \label{conv_x}\\
    &\lim_{k\rightarrow\infty}\|\vs^{k+1}-\vs^k\|=0, \label{conv_s}
\end{align}
due to $\theta>\frac{3}{4}$ and $\vM+\vN\succ\vzero.$ Moreover, from~\eqref{iter:b}, we get
\begin{equation}
    \lim_{k\rightarrow\infty}\|\zeta^{k+1}-\zeta^k\|=0 \label{conv_zeta}.
\end{equation}
From the descent inequality, the nonnegative sequence $\{\Phi^k\}$ is nonincreasing, so it converges to a nonnegative constant, which implies $(\vs^k,\vx^k,\zeta^k)$ is bounded in $\RR^m\times\RR^n\times\RR^n$. Due to the compactness, there is a subsequence, $\{(\vs^{k_j}, \vx^{k_j}, \zeta^{k_j})\}$ converging to $(\vs^\star,\vx^\star,\zeta^\star)$.
 
 The nonexpansiveness of $\prox_{rg}$ and $\prox_{\frac{\lambda}{r}h^*}$ implies they are continuous and from~\eqref{conv_x},~\eqref{conv_s} and~\eqref{conv_zeta}, 
 $$\begin{aligned}
 &\ \lim_{j\rightarrow\infty}\|(\vs^{k_j+1},\vx^{k_j+1},\zeta^{k_j+1})-(\vs^\star,\vx^\star,\zeta^\star)\|\\
 \leq&\ \lim_{j\rightarrow\infty}\|(\vs^{k_j+1},\vx^{k_j+1},\zeta^{k_j+1})-(\vs^{k_j}, \vx^{k_j}, \zeta^{k_j})\|+\|(\vs^{k_j},\vx^{k_j},\zeta^{k_j})-(\vs^\star,\vx^\star,\zeta^\star)\|\\
 =&\ 0.
 \end{aligned}$$
So, $(\vs^\star,\vx^\star,\zeta^\star)$ is a fixed point of the iteration~\eqref{algo}.
 
By choosing $(\vs^*,\vx^*,\zeta^*)=(\vs^\star,\vx^\star,\zeta^\star),$ we have
 $$\lim_{j\rightarrow\infty}\Phi^{k_j}=0.$$
Hence
 $$\lim_{j\rightarrow\infty}(\vs^k,\vx^k,\zeta^k)=(\vs^\star,\vx^\star,\zeta^\star).$$
Lastly, from Lemma~\ref{lem}, $(\vs^\star,\vx^\star)$ is a solution of the problem~\eqref{pb2}.
\qed
\end{proof}

Due to the relation discussed in Section~\ref{connect:afba} and \ref{connect:pd3o}, we also prove the convergence of AFBA and PD3O under the relaxed condition. % When $f$ is linear, we recover the improved result of L-ADMM in~\cite{he2020optimally}. 
When $f=0$, the relation in Section~\ref{connect:cp} implies the improved convergence result of Chambolle-Pock.
\begin{corollary}[Chambolle-Pock]
Suppose $r>0$ and $\lambda<\frac{4}{3\sigma^2}$, we have the following results hold,
\begin{enumerate}
\item Let $\{(\vs^k,\vx^k,\zeta^k)\}$ be generated by the iteration~\eqref{algo}, then $(\vs^k,\vx^k)$ converges to a solution of the problem~\eqref{pb2} with linear $f$.
\item Let $\{(\vs_{3}^k,\vx_3^k)\}$ be generated by Chambolle-Pock~\eqref{cp}, then it converges to a solution of the problem~\eqref{pb2} with $f=0$.
\end{enumerate}
\end{corollary}
\begin{proof}
~\begin{enumerate}
\item When $f$ is linear, the Lipschitz constant of gradient is $0$. We can set $L=\epsilon>0$, then by Theorem~\ref{thm} if $r<\frac{4\theta-3}{2\theta-1}\frac{2}{\epsilon}$ and $\lambda\leq\frac{1}{\theta\sigma^2},$
$$\lim_{k\rightarrow\infty}(\vs^k,\vx^k)=(\vs^\star,\vx^\star),$$
where $(\vs^\star,\vx^\star)$ is a saddle-point solution.

Let $\epsilon\rightarrow0$ and $\theta\rightarrow 3/4$, then the proof is complete.
\item From the sequence relation~\eqref{base_cp} and the above argument, we have
$$\lim_{k\rightarrow\infty}(\vs^k_3,\vx^k_3)=\lim_{k\rightarrow\infty}(\vs^k,\vx^{k-1}-r\vA^\top(\vs^k-\vs^{k-1}))=(\vs^\star,\vx^\star).$$
\end{enumerate}\qed
\end{proof}

The comparison of aforementioned algorithms with improved conditions on stepsizes is summarized in Table~\ref{table2}, where the results of the algorithm in bold are proved in this paper.

\begin{table}[!ht]
\centering
\begin{tabular}{lllllll}
\hhline{-------}
                                                  &  & $f$        &  & $g$    &  & $r, \lambda$                                                                                   \\ \hhline{=======}
\textbf{Chambolle-Pock}                                     &  & $0$        &  & convex &  & $r>0$, $ \lambda<4/(3\sigma^2)$                                                     \\
PAPC(PDFP$^2$O)                               &  & $L$-smooth &  & $0$    &  & $rL/2<1$, $\lambda<4/(3\sigma^2)$                                          \\
Condat-Vu                                        &  & $L$-smooth &  & convex &  & $\lambda\sigma^2+rL/2\leq 1$                                                      \\
PDFP                                            &  & $L$-smooth &  & convex &  & $rL/2<1$, $\lambda\sigma^2(\vA)\leq 1$                                                   \\
\textbf{AFBA}                                            &  & $L$-smooth &  & convex &  & $rL/2<\Gamma$,  $\theta\lambda\sigma^2\leq 1$        \\
\textbf{PD3O}                                             &  & $L$-smooth &  & convex &  & $rL/2<\Gamma$, $  \theta\lambda\sigma^2\leq 1$                                               \\
\textbf{The Base Algorithm} &  & $L$-smooth &  & convex &  & $rL/2<\Gamma$, $  \theta\lambda\sigma^2\leq 1$ %$\frac{rL}{2}<\frac{4\theta-3}{2\theta-1}, \theta\lambda\sigma^2(\vA)\leq 1$
\\ \hline
\end{tabular}
\caption{The comparison of the requirement of stepsizes to guarantee the convergence of some primal-dual algorithms. $L$-smooth means the function is convex and has a $L$-Lipschitz continuous gradient. $\Gamma\coloneqq (4\theta-3)/(2\theta-1)$, where $\theta$ is an arbitrary number in $(3/4,1].$ }
\label{table2}
\end{table}

\subsection{Tightness of Upper Bound for Stepsizes}\label{tight}
In this section, we provide a simple example to show that the upper bound for $\lambda$ can not be relaxed further. This example is more general than the one provide in~\cite{he2020optimally}. We consider the following saddle point problem 
$$\min_{\vx\in\RR^n}\max_{\vs\in\RR^m}~\langle \vA\vx,\vs\rangle.$$
Any $(\vx,\vs)$ such that $\vA\vx=\vzero$ and $\vA^\top\vs=\vzero$ is a solution. 
For this special problem, one iteration becomes 
\begin{align}
    \vs^{k+1} =&~ {\lambda\over r}\vA\vx^k+(\vI-\lambda\vA\vA^\top)\vx^k,\\
    \vx^{k+1}=&~\vx^k-r\vA^\top\vs^{k+1}=(\vI-\lambda\vA^\top\vA)\vx^k-r \vA^\top(\vI-\lambda\vA\vA^\top)\vs^k,
\end{align}
We can rewrite it as 
\begin{align}
    \begin{bmatrix}
    \vs^{k+1}\\\vA\vx^{k+1}
    \end{bmatrix}=
    \begin{bmatrix}
    \vI-\lambda\vA\vA^\top &{\lambda\over r}\vI\\-r\vA\vA^\top(\vI-\lambda\vA\vA^\top)&\vI-\lambda\vA\vA^\top
    \end{bmatrix}
        \begin{bmatrix}
    \vs^{k}\\\vA\vx^{k}
    \end{bmatrix}.
\end{align}
Therefore, to make the algorithm converge, the eigenvalues of the matrix can not have magnitude larger than 1. Because $\vA\vA^\top$ is symmetric, we only need to consider the $2\times 2$ matrix 
\begin{align}
    \begin{bmatrix}
        1-\lambda\theta & {\lambda\over r} \\
        -r \theta(1-\lambda \theta) & 1-\lambda \theta
    \end{bmatrix},
\end{align}
where $\theta$ is the eigenvalues of $\vA\vA^\top$. Its two eigenvalues are 
$$1-\lambda\theta\pm\sqrt{-\lambda\theta(1-\lambda\theta)}.$$

We consider different cases for $\lambda\theta$. If $\lambda\theta< 1$, both eigenvalues are complex numbers, and their magnitude is 
$\sqrt{(1-\lambda\theta)^2+\lambda\theta(1-\lambda\theta)}\leq 1$. If $\lambda\theta=1$, both eigenvalues are zero. When $\lambda\theta>1$, both eigenvalues are real number. The eigenvalue $1-\lambda\theta+\sqrt{-\lambda\theta(1-\lambda\theta)}<1-\lambda\theta+\sqrt{\lambda\theta\lambda\theta}=1.$ The other eigenvalue is $1-\lambda\theta-\sqrt{\lambda\theta(\lambda\theta-1)}$. To make sure that its magnitude is less than one, we need $1-\lambda\theta-\sqrt{\lambda\theta(\lambda\theta-1)}>-1$, that is $\lambda\theta<4/3.$ The condition for the convergence with any initial value is $\lambda\theta<4/3$ for all eigenvalues of $\vA\vA^\top$, that is $\lambda\sigma^2<4/3$.

This example shows that the condition $\lambda\sigma^2<4/3$ can not be relaxed further for Chambolle-Pock.

\section{Numerical Experiments}\label{sec:exp}
In this section, we demonstrate the performance of several primal-dual algorithms under the relaxed condition and compare their results with existing ones. More specifically, we use PD3O and AFBA to solve the fused LASSO (least absolute shrinkage and selection operator) and Chambolle-Pock to solve LASSO to show their convergence with different combinations of $r$ and $\lambda$. 

\subsection{The fused LASSO}
The fused LASSO problem (see, e.g.,~\cite{tibshirani2005sparsity}) is formulated as
\begin{equation}\label{flasso}
    \Min_{\vx\in\RR^{10000}}\ \frac{1}{2}\|\vK\vx-\vb\|^2+\mu_1\|\vB\vx\|_1+\mu_2\|\vx\|_1
\end{equation}
where $\vK\in\RR^{500\times 10000}$. The two penalty parameters $\mu_1$ and $\mu_2$ are  set to $200$ and $20$, respectively. The $ith$ row of $\vB\in\RR^{9999\times 10000}$ has $-1$ on the $i$th column, $1$ on the $i+1$th column, and $0$ on other columns.

We let $f(\vx)=\frac{1}{2}\|\vK\vx-\vb\|^2,$ $g(\vx)=\mu_2\|\vx\|_1$ and $h(\vB\vx)=\mu_1\|\vB\vx\|_1$. Then the primal-dual form is
\begin{equation}
    \min_{\vx\in\RR^{10000}}\max_{\vs\in\RR^{9999}}\ \frac{1}{2}\|\vK\vx-\vb\|^2+\mu_2\|\vx\|_1+\dotp{\vB\vx,\vs}-\mu_1\chi_{B_{\infty}}\left(\frac{\vs}{\mu_1}\right),
\end{equation}
where $B_{\infty}$ is the closed unit ball in $\ell_{\infty}$ norm and $\chi_{B_{\infty}}$ is the indicator function over $B_{\infty}$. 

We generate a sparse vector $\vx_{\text{True}}$ with $50$ nonzero entries and a dense matrix $\vA$ whose each entry is independently sampled from the standard normal distribution. The response vector $\vb$ is obtained by adding Gaussian noise to $\vA\vx_{\text{True}}$. We calculate the estimated optimal solution by running $10,000$ steps PD3O for the problem and get the estimated optimal function value $f^*$. 

We set the default parameters $r=1/\sigma^2(\vK)$, $\lambda = 1/\sigma^2(\vB)$ and consider several choices of stepsizes in the two scenarios: (1) $\frac{r L}{2}<\frac{4\theta-3}{2\theta-1}, \lambda\leq\frac{1}{\theta\sigma^2};$ (2) $\frac{rL}{2}<1, \lambda<\frac{4}{3\sigma^2}.$  The first scenario obeys the relaxed condition shown in this paper, while the other one may violate the condition. In Fig.~\ref{fig:flasso}, the left figure shows the result with $\theta = 1/1.19$ and $4/5$ for the first scenario. In this figure, we compare four choices of the stepsizes: (1) the default parameter $(r,\lambda)$; (2) we fix the primal stepsize $r$ and choose a small $\theta=1/1.19$ to obtain a large $\lambda$. The new parameter is $(r,1.19\lambda)$; (3) Choose a smaller $\theta=0.8$ and decrease the primal stepsize $r$ only; (4) Choose the same $\theta=0.8$ and increase the parameter $\lambda$ to its upper bound. The right figure in Fig.~\ref{fig:flasso} compare the convergence of algorithms under the second scenario with a larger $\lambda$. Note that, the settings with $1.3\lambda$ do not satisfy the condition in this paper. We observe that in either scenario, the primal stepsize dominates the convergence of algorithms and a slightly larger $\lambda$ has little effect on the algorithm.

\begin{figure*}[!ht]
\centering
\begin{minipage}[t]{0.494\textwidth}
    \centering
\includegraphics[width=1.0\textwidth]{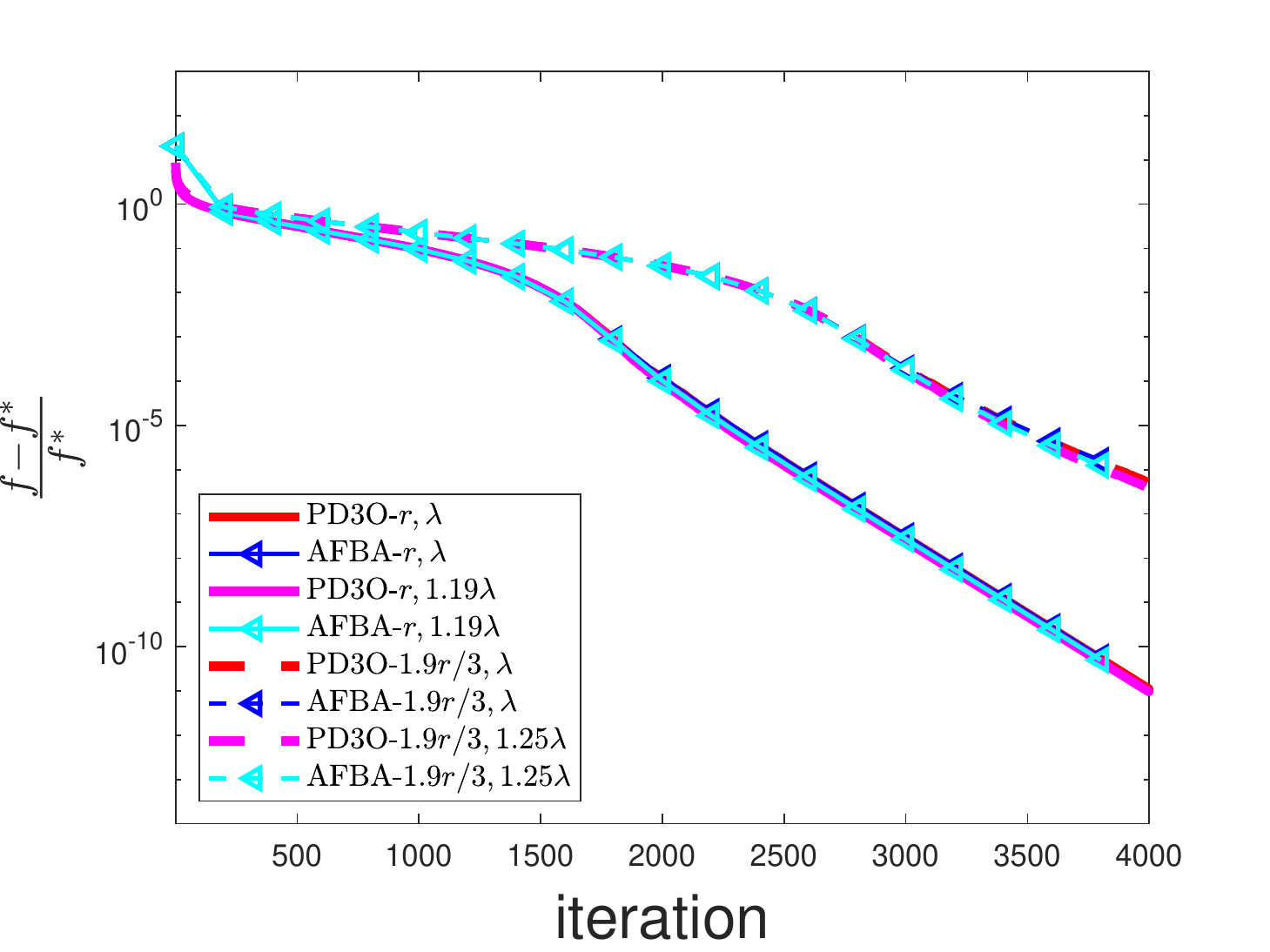}
    % \vspace{-0.35in}
    \label{fig:flasso_1}
\end{minipage}
\begin{minipage}[t]{0.494\textwidth}
    \centering
    \includegraphics[width=1.0\textwidth]{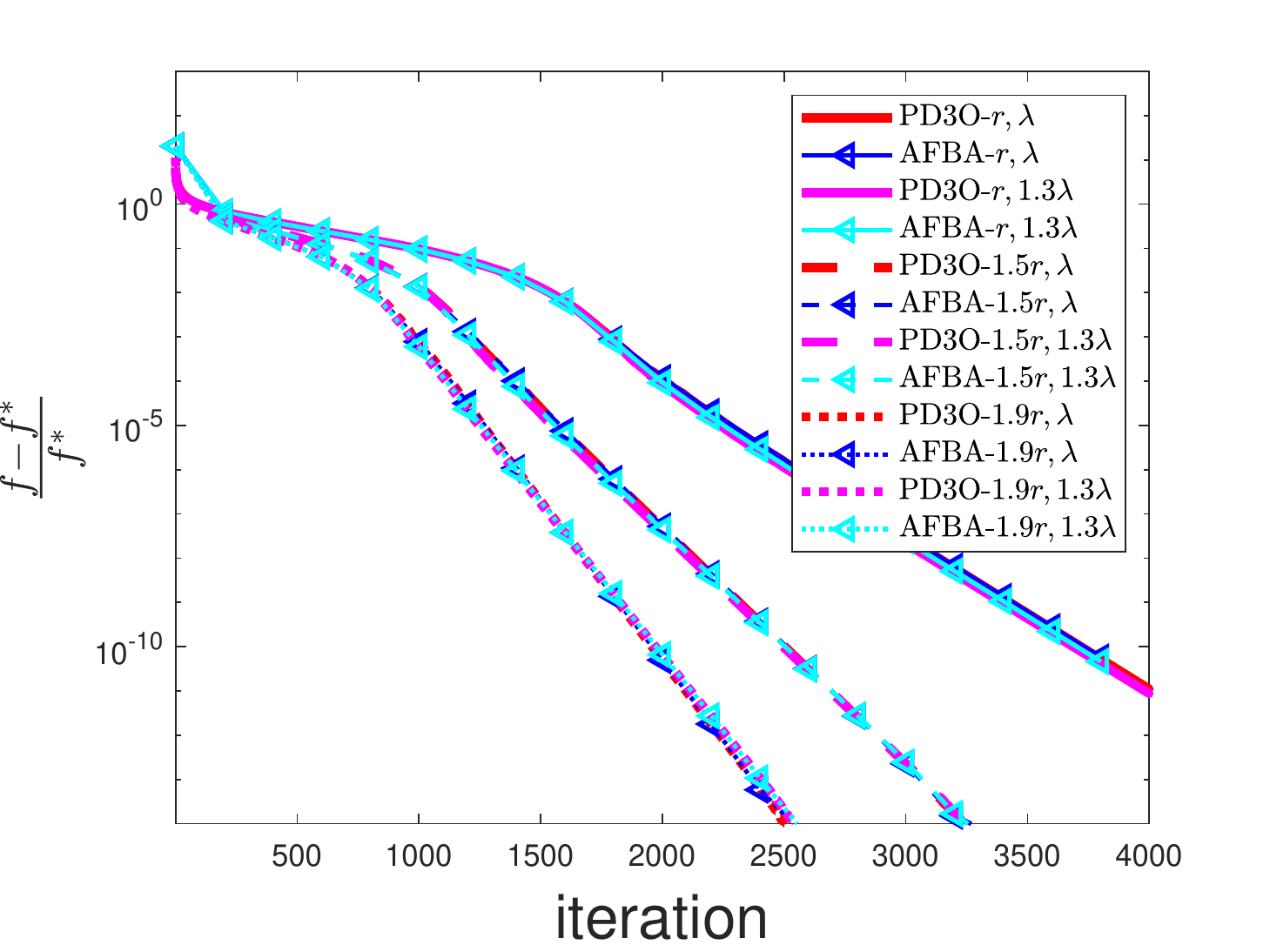}
    % \vspace{-0.35in}
    \label{fig:flasso_2}
\end{minipage}
\caption{The comparison of the performance of PD3O and AFBA with different parameters. In both figures, the fixed parameters $r$ and $\lambda$ are set to $1/\sigma^2(\vK)$ and $1/\sigma^2(\vB)$, respectively.}
% \caption{Smooth case.}
\label{fig:flasso}
\end{figure*}

However, the benefit of larger $\lambda$ appears when $\vK$ has a full column rank. We consider the same fused LASSO problem with  $\vK\in\RR^{2500\times2500}$. We conduct the experiments for two cases, randomly generated $\vK$ and $\vK=\vI.$ The problem setting is changed to $\mu_1 =5$ and $\mu_2=1/5$. The true solution is generated in the same way with $25$ nonzero entries.  As shown in Fig.~\ref{fig:flasso2}, both top and bottom figures indicate $10-20\%$ acceleration of convergence when $\lambda$ is increased.    
The numerical result on the second scenario also suggests that the general constraint on $r$ and $\lambda$ shown in this paper may not be tight. 

\begin{figure*}[!ht]
\centering
\subfloat[Random $\vK\in\RR^{2500\times 2500}$ ]{
\begin{minipage}[t]{0.494\textwidth}
    \centering
\includegraphics[width=1.0\textwidth]{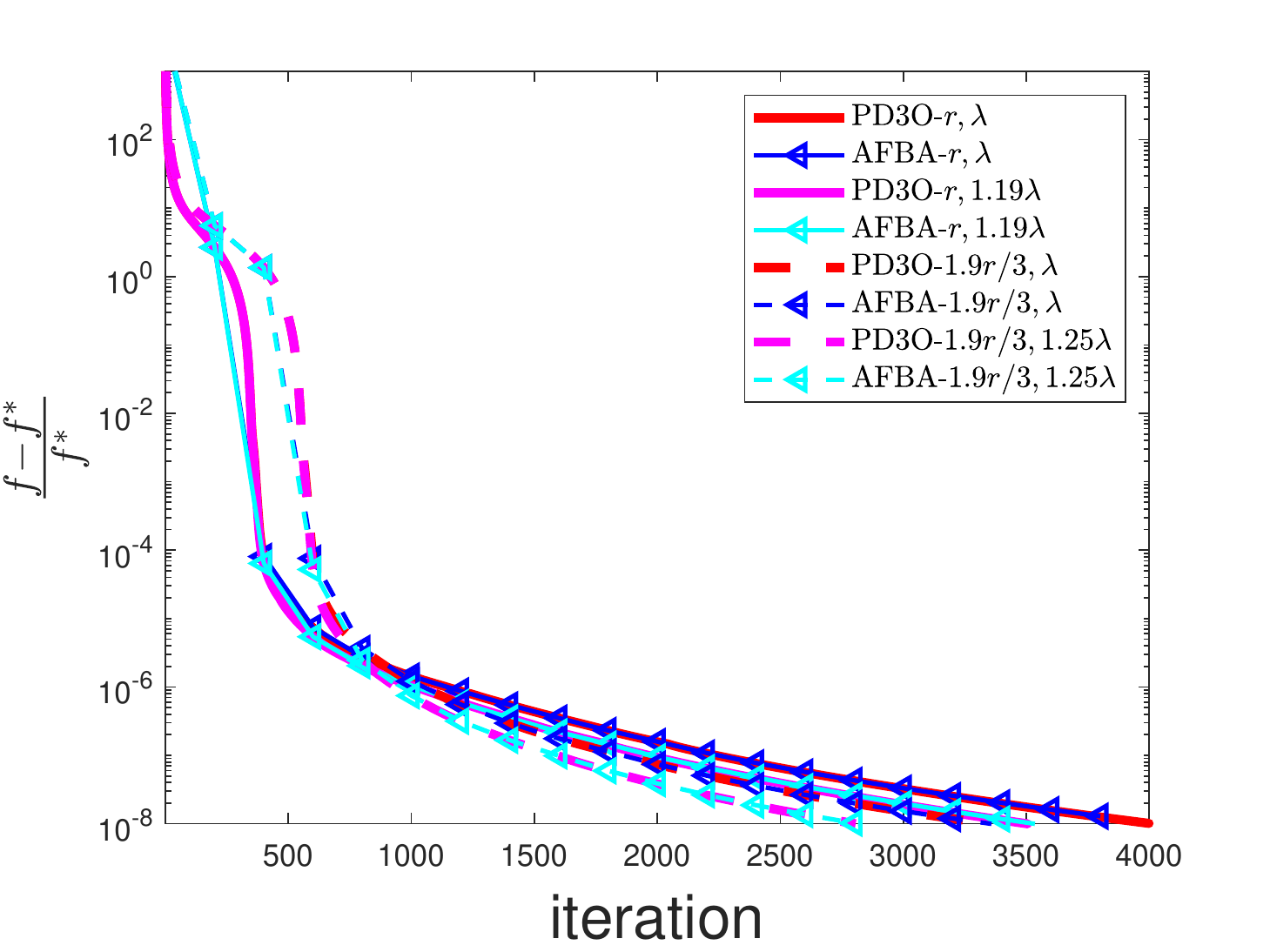}
    % \vspace{-0.35in}
    \label{fig:flasso2_1}
\end{minipage}
\begin{minipage}[t]{0.494\textwidth}
    \centering
    \includegraphics[width=1.0\textwidth]{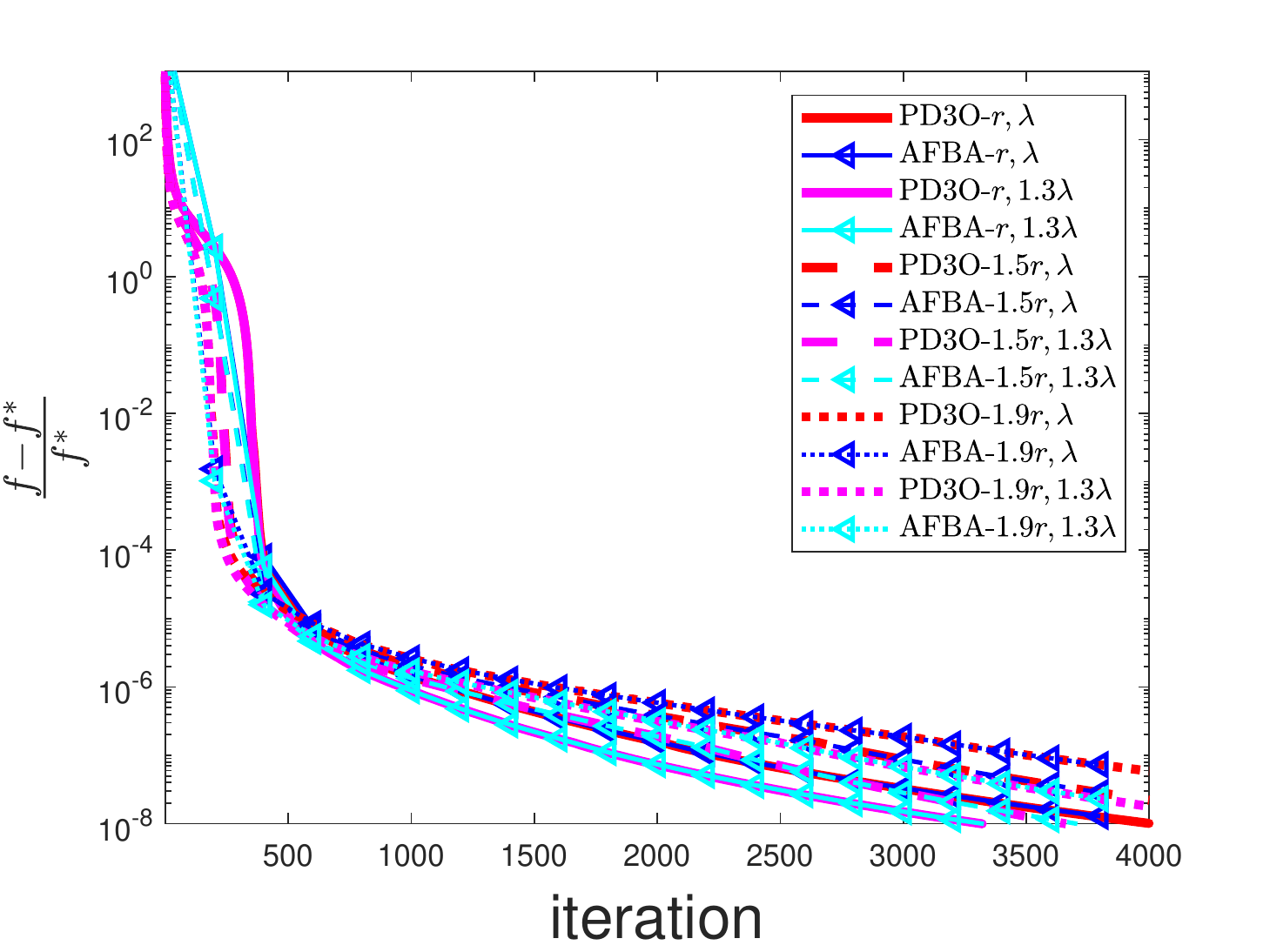}
    % \vspace{-0.35in}
    \label{fig:flasso2_2}
\end{minipage}
}
\\
\subfloat[$\vK=\vI\in\RR^{2500\times2500}$]{
\begin{minipage}[t]{0.494\textwidth}
    \centering
\includegraphics[width=1.0\textwidth]{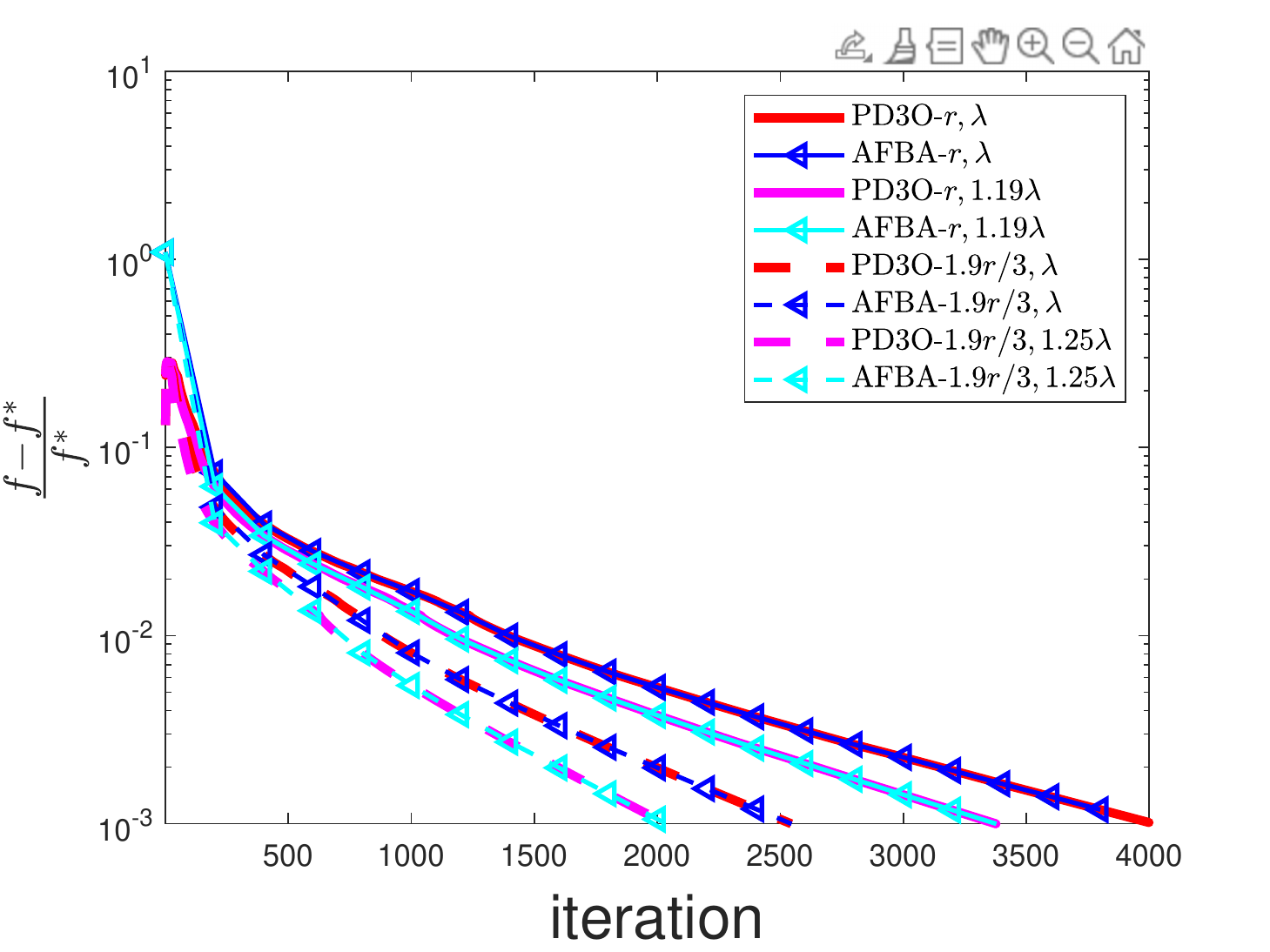}
    % \vspace{-0.35in}
    \label{fig:flasso2_3}
\end{minipage}
\begin{minipage}[t]{0.494\textwidth}
    \centering
    \includegraphics[width=1.0\textwidth]{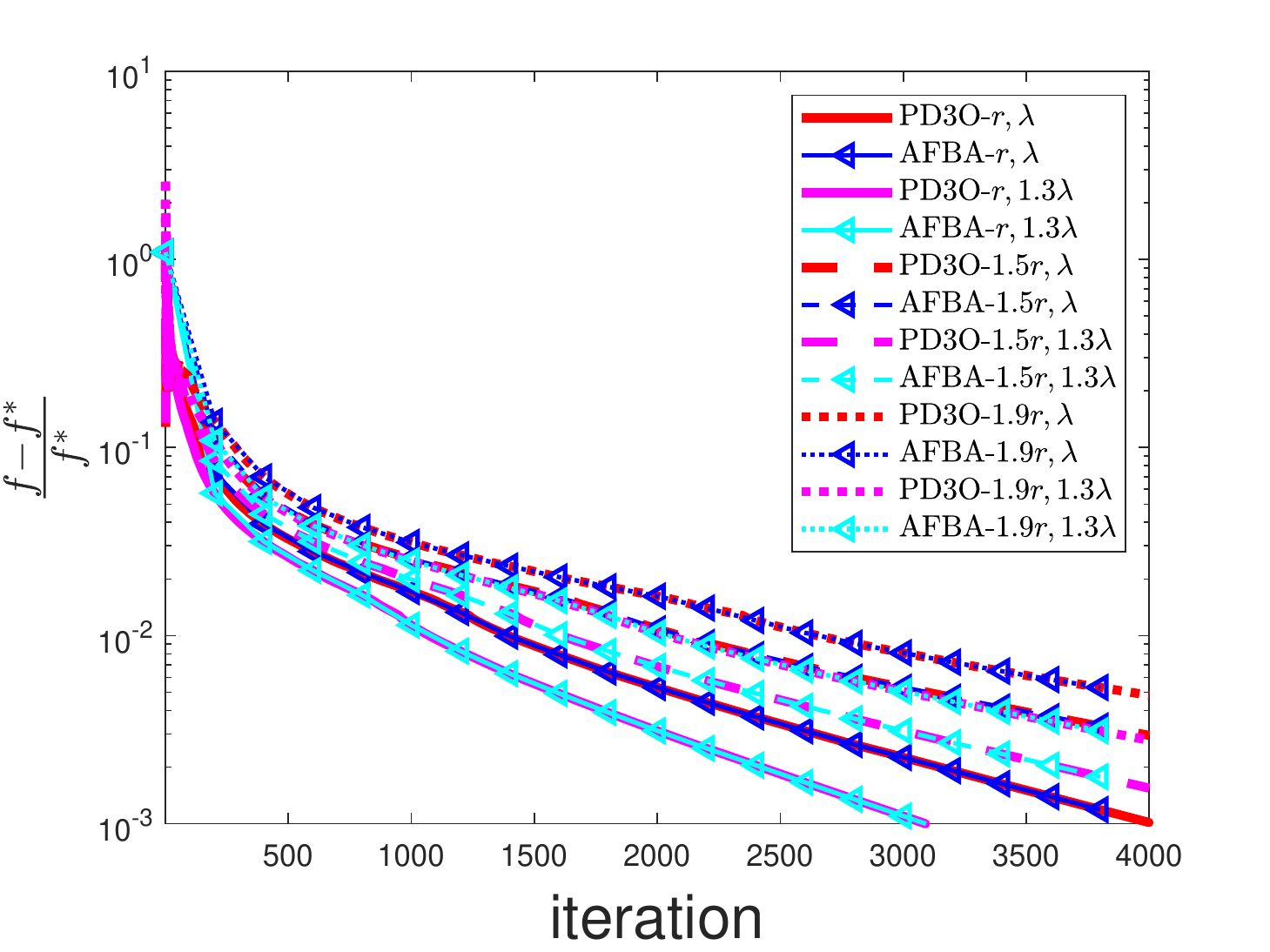}
    % \vspace{-0.35in}
    \label{fig:flasso2_4}
\end{minipage}
}

\caption{The comparison of the performance of PD3O and AFBA with different parameters. In all figures, the fixed parameters $r$ and $\lambda$ are set to $1/\sigma^2(\vK)$ and $1/\sigma^2(\vB)$, respectively.
}
% \caption{Smooth case.}
\label{fig:flasso2}
\end{figure*}

%\iffalse One hypothesis made from the numerical result is that the forward step of smooth function $f$ dominates the convergence of the algorithm so the increment of stepsize on backward step is negligible. In order to verify this, we solve the same problem using Chambolle-Pock.     

%\subsection{The fused LASSO (Two Operators)}
%The solution of the fused LASSO~\eqref{flasso} can also be derived from the following $1$-D TV denoising problem. 
%\begin{equation}
%    \Min_{\vx\in\RR^{10000}}\ \frac{1}{2}\|\vK\vx-\vb\|^2+\mu_1\|\vB\vx\|_1,
%\end{equation}
%where the only difference here is removing $\ell_1$ penalty on $\vx$. As shown in~\cite[Lemma A.1]{friedman2007pathwise}, the solution $\vx^*$ of the problem~\eqref{flasso} is obtained from the solution $\vx^\star$ of the above problem by the relation???
%\begin{equation*}
%    \vx^*_i=\left\{\begin{aligned}
%        \vx^\star_i-\mu_2\sign(\vx^\star_i),\ &\  \text{if}\ |\vx_i^\star|>\mu_2,\\
 %       0,\ &\  \text{otherwise}.
%    \end{aligned}\right.
%\end{equation*}

%We let $g(\vx)=\frac{1}{2}\|\vK\vx-\vb\|$ and $h(\vB\vx)=\mu_1\|\vB\vx\|_1$ and consider the primal-dual form as
%\begin{equation}
%    \min_{\vx\in\RR^{10000}}\max_{\vs\in\RR^{9999}}\ \frac{1}{2}\|\vK\vx-\vb\|^2+ 
%\end{equation}
%\fi

\subsection{LASSO}
We consider the following LASSO problem (see, e.g.,~\cite{tibshirani1996regression})
\begin{equation}
    \Min_{\vx\in\RR^{5000}}\ \frac{1}{2}\|\vK\vx-\vb\|^2+\mu\|\vx\|_1
\end{equation}
where $\vK\in\RR^{500\times 5000}$ and $\mu=200$ is a penalty parameter. We let $g(\vx)=\mu\|\vx\|_1$ and $h(\vK\vx)=\frac{1}{2}\|\vK\vx-\vb\|^2$ and consider the primal-dual form as
\begin{equation}
    \min_{\vx\in\RR^{5000}}\max_{\vs\in\RR^{500}}\ \mu\|\vx\|_1+\dotp{\vK\vx,\vs}-\frac{1}{2}\|\vs\|^2-\dotp{\vb,\vs}.
\end{equation}

The data are generated in a similar way to the previous experiment.  The optimal solution $\vx^*$ is calculated by performing 10,000 iterations of Chambolle-Pock. Fig.~\ref{fig:lasso} shows the effect of the relaxed conditions on the convergence of Chambolle-Pock in terms of two different residual measures. The default dual parameter $\lambda$ is set to  $1/\sigma^2(\vK)$ and the relaxed one is $1.32\lambda$. The primal stepsize $r$ is chosen from $\{0.001, 0.005, 0.01, 0.05\}$. 

\begin{figure*}[!ht]
\centering
\begin{minipage}[t]{0.494\textwidth}
    \centering
\includegraphics[width=1.0\textwidth]{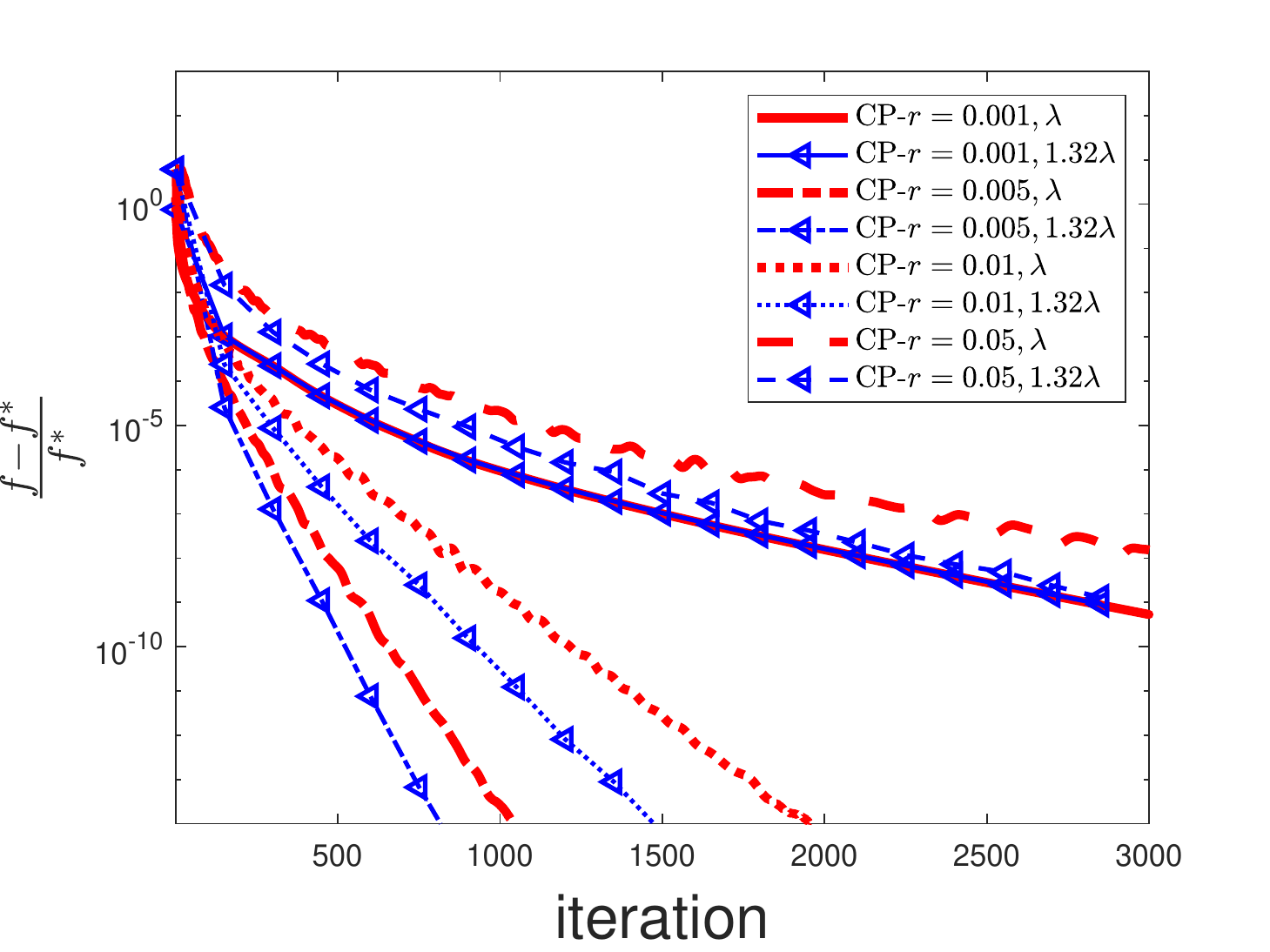}
    % \vspace{-0.35in}
    \label{fig:lasso_energy}
\end{minipage}
\begin{minipage}[t]{0.494\textwidth}
    \centering
    \includegraphics[width=1.0\textwidth]{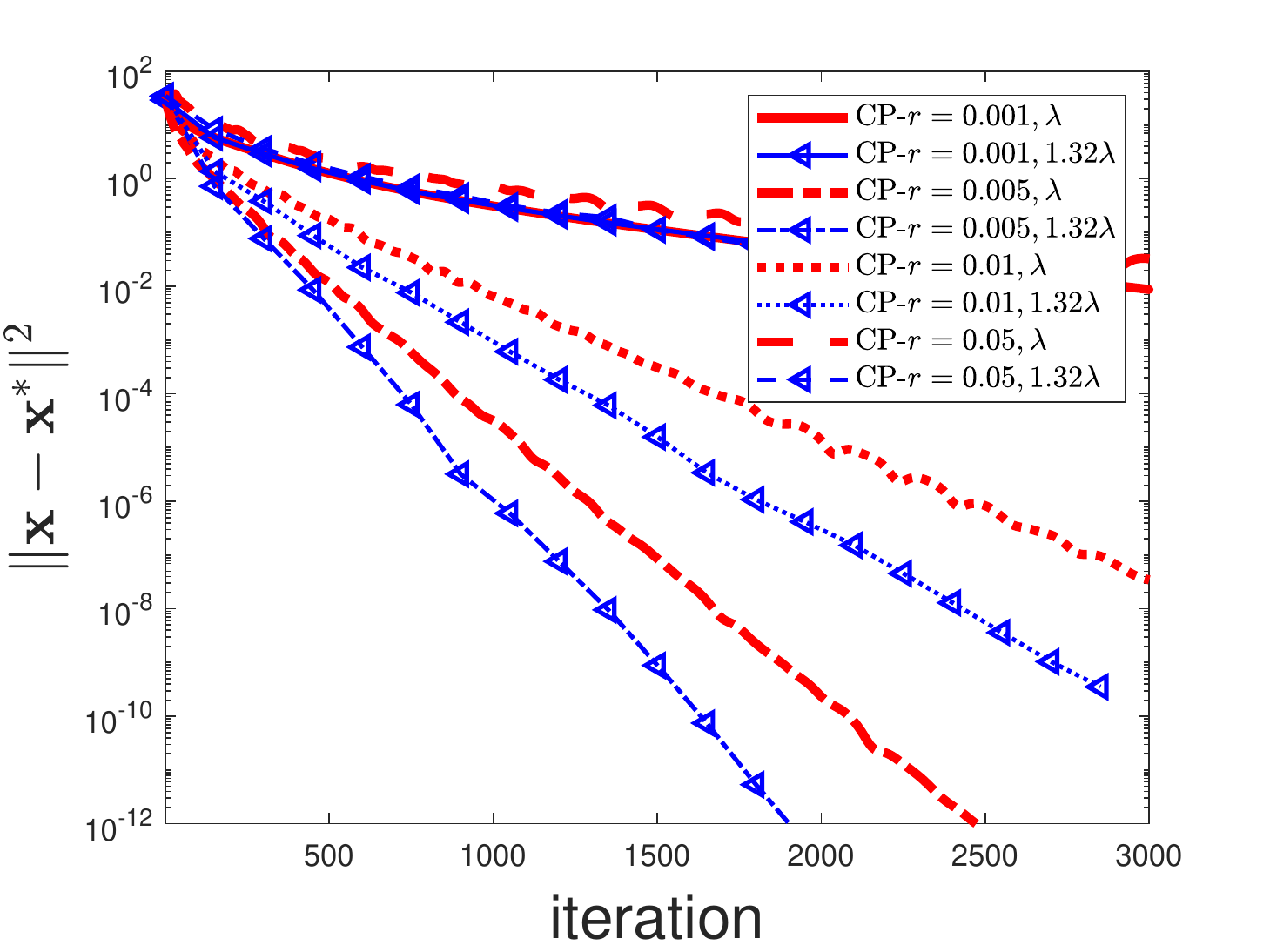}
    % \vspace{-0.35in}
    \label{fig:lasso_LS}
\end{minipage}
\caption{The comparison of the performance of Chambolle-Pock with different parameters on LASSO in terms of the function values and distance to the optimal solution. In both figures, the fixed parameter $\lambda$ is set to $1/\sigma^2(\vA)$.}
% \caption{Smooth case.}
\label{fig:lasso}
\end{figure*}

From Fig.~\ref{fig:lasso}, we observe that by choosing the acceptable range of $r$, the larger value of $\lambda$ makes the algorithm converge faster up to $20-30\%$ acceleration.  We also note that the relaxed $\lambda$ fails to speed up the algorithm and the residual curves overlap when $r$ is set to a very small number. This is possibly due to the small progression of the primal variable at each step. The experiment on extreme values of $r$ verifies the explanation as shown in Fig.~\ref{fig:lasso2}. We use different colors to differentiate the result of the extreme large values from the result of the extreme small values. This figure also suggests that balanced primal and dual stepsizes is needed. Though how to choose good primal and dual stepsizes is out of the focus of this paper, the relaxed condition in this paper provided theoretic guarantee to increase one or both stepsizes.

\begin{figure*}[!ht]
\centering
\begin{minipage}[t]{0.494\textwidth}
    \centering
\includegraphics[width=1.0\textwidth]{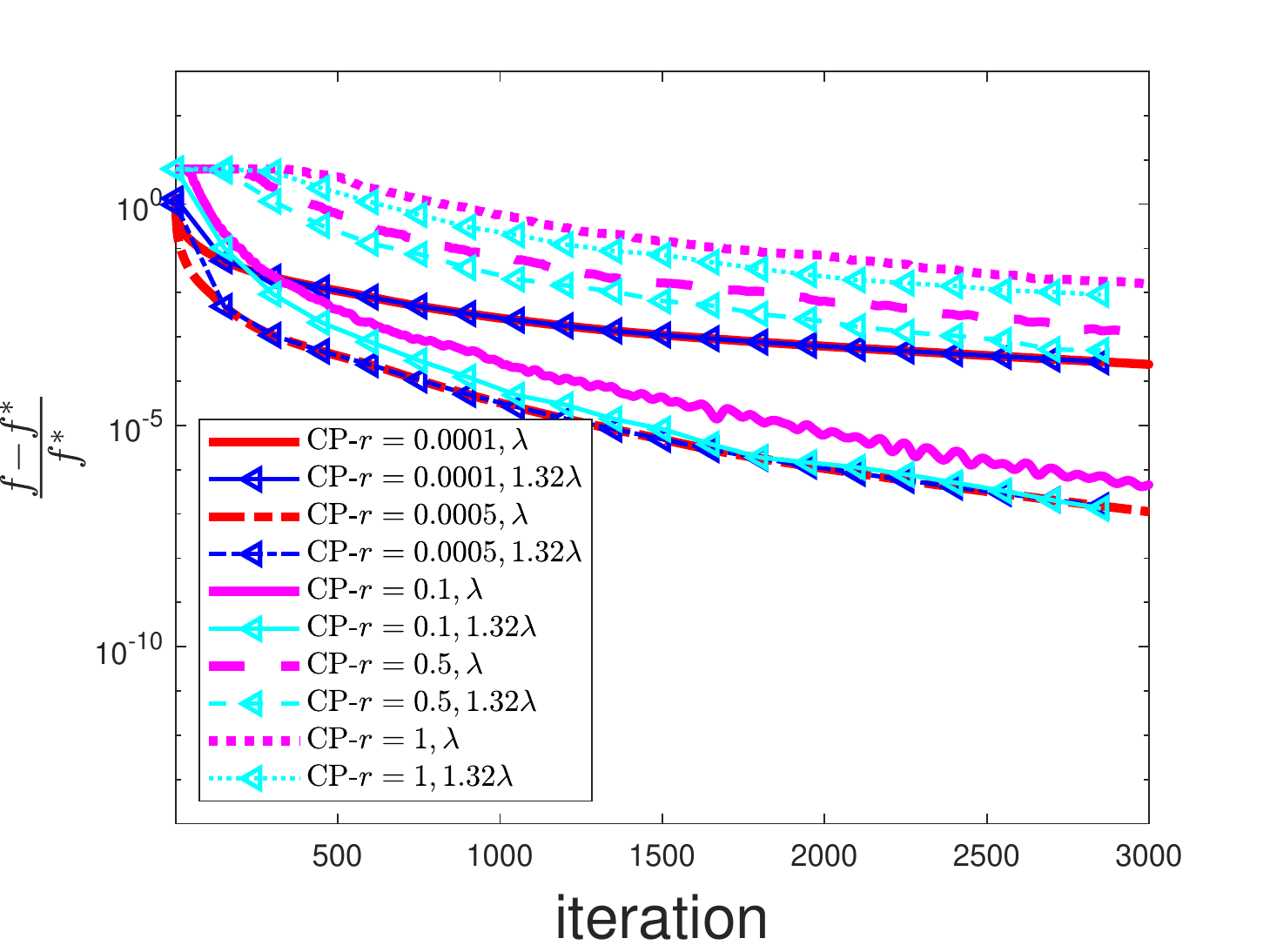}
    % \vspace{-0.35in}
    \label{fig:lasso_energy2}
\end{minipage}
\begin{minipage}[t]{0.494\textwidth}
    \centering
    \includegraphics[width=1.0\textwidth]{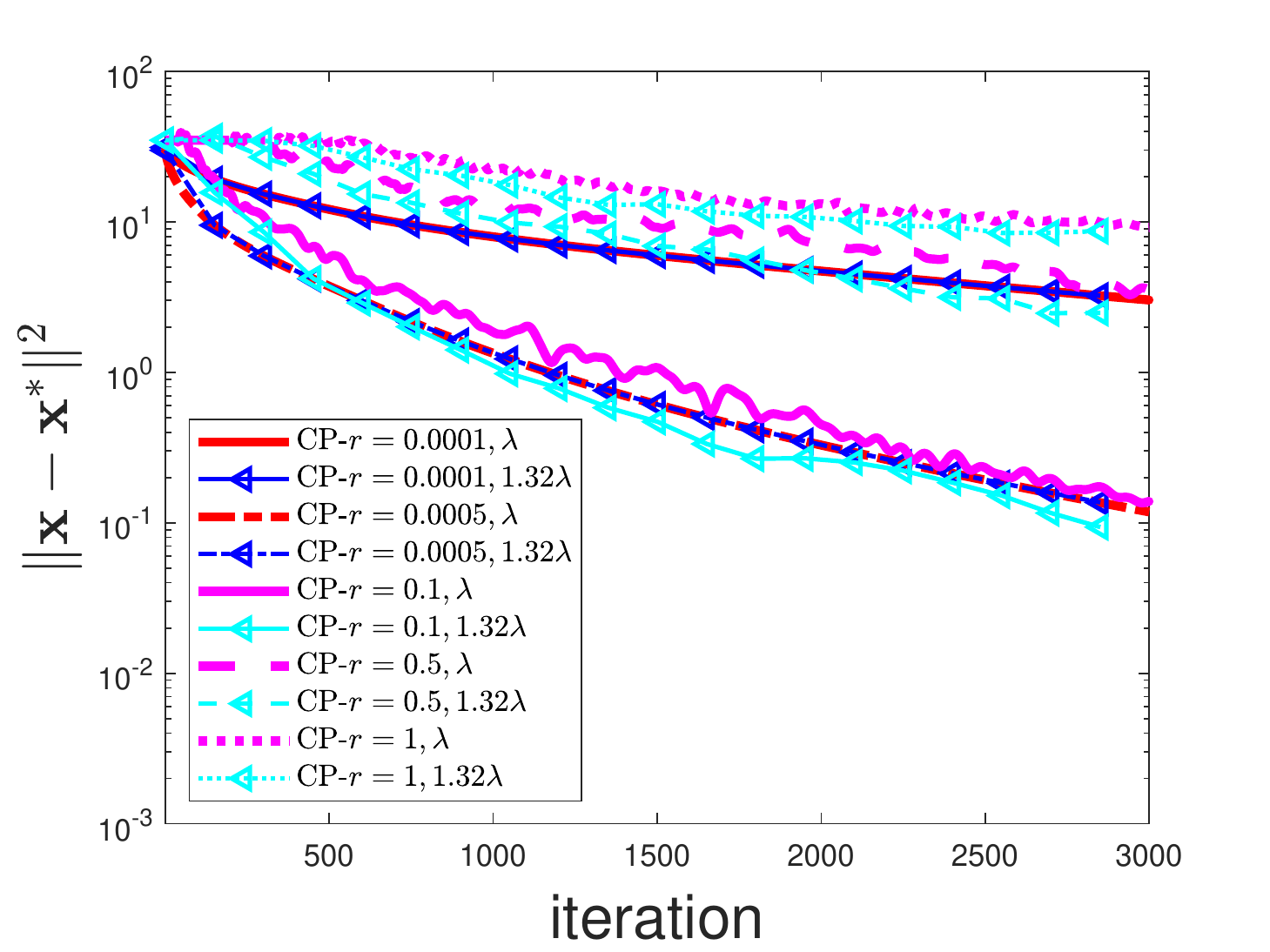}
    % \vspace{-0.35in}
    \label{fig:lasso_LS2}
\end{minipage}
\caption{The comparison of performance of Chambolle-Pock with different parameters chosen from extreme values. In both figures, the fixed parameter $\lambda$ is set to $1/\sigma^2(\vA)$.}
% \caption{Smooth case.}
\label{fig:lasso2}
\end{figure*}

\section{Conclusions}
In this paper, we use a base algorithm to build the connection between some primal-dual algorithms. Then we prove the convergence of the base algorithm under a relaxed condition for the primal and dual stepsizes. It implies a possible choice of larger dual stepsize with a corresponding conservative primal stepsize. Chambolle-Pock, as a special case of the base algorithm, can take the dual stepsize up to $4/3$ of the original one without a compromise on the primal stepsize. The numerical experiments indicates the acceleration for tested algorithms under the relaxed condition. The benefit for Chambolle-Pock is more prominent. The condition for Chambolle-Pock is also proved to be optimal and can not be relaxed. However, how to relax the condition further for the base algorithm is still an open problem. 

%Finally, the numerical results for PD3O and AFBA imply the improved condition on stepsizes can be further relaxed, which will be left for future consideration.

\begin{acknowledgements}
This work is partially supported by the NSF grant DMS-2012439.
\end{acknowledgements}

%References
\bibliographystyle{unsrt}
\bibliography{ref}

\end{document}